\documentclass[11pt, leqno]{article}
\usepackage{amsmath,amsfonts,amssymb,wasysym,mathrsfs}
\usepackage[latin1]{inputenc}
\usepackage{shapepar}
\usepackage{graphicx}
\usepackage[T1]{fontenc}
\usepackage{frcursive}

\usepackage{hyperref}
\usepackage{ifpdf} 
\usepackage{color}
\DeclareMathAlphabet{\mathpzc}{OT1}{pzc}{m}{it}
\newcommand{\R}{{\mathbb R}}
\newenvironment{frcseries}{\fontfamily{frc}\selectfont}{}

\newcommand{\be}[1]{\begin{equation}\label{#1}}
\newcommand{\ee}{\end{equation}}

\newcommand{\mathfrc}[1]{\text{\textfrc{#1}}}
\newcommand{\textfrc}[1]{{\frcseries#1}}

\newcommand{\prf}{\par\smallskip\noindent{\sl Proof. \/}}
\newcommand{\finprf}{\unskip\null\hfill$\;\square$\vskip 0.3cm}

\newenvironment{proof}{\prf}{\finprf}
\newtheorem{theorem}{Theorem}[section]
\newtheorem{lemma}{Lemma}[section]

\newtheorem{proposition}[theorem]{Proposition}
\newtheorem{remark}[theorem]{Remark}
\newtheorem{definition}{Definition}[section]

\numberwithin{equation}{section}
\newcommand{\nc}{\normalcolor}

\def\qed{\,\unskip\kern 6pt \penalty 500
\raise -2pt\hbox{\vrule \vbox to8pt{\hrule width 6pt
\vfill\hrule}\vrule}\par}
\definecolor{darkblue}{rgb}{0.05, .05, .65}
\definecolor{darkgreen}{rgb}{0.1, .65, .1}
\definecolor{darkred}{rgb}{0.8,0,0}
\begin{document}
\title{\textbf{Symmetrization for fractional elliptic problems: a direct approach}\\[7mm]}

\author{\Large Vincenzo Ferone\footnote{Dipartimento di Matematica e Applicazioni ``Renato Caccioppoli'', Universit\`a degli Studi di
Napoli Federico II, 80143 Napoli, Italy. \ E-mail: {\tt ferone@unina.it}} \ and \ Bruno Volzone\footnote{Dipartimento di Scienze e Tecnologie, Universit\`a degli Studi di
Napoli ``Parthenope'', 80133 Napoli, Italy. \ E-mail: {\tt bruno.volzone@uniparthenope.it}}}

\date{} 

\maketitle

\begin{abstract}
We provide new direct methods to establish symmetrization results in the form of mass concentration (\emph{i.e.} integral) comparison for fractional elliptic equations of the type $(-\Delta)^{s}u=f$ $(0<s<1)$ in a bounded domain $\Omega$, equipped with homogeneous {Dirichlet }boundary conditions. The classical pointwise Talenti rearrangement inequality in \cite{Talenti1} is recovered in the limit $s\rightarrow1$. Finally, explicit counterexamples constructed for all $s\in(0,1)$ highlight that the same pointwise estimate cannot hold in a nonlocal setting, thus showing the optimality of our results.
\end{abstract}

\setcounter{page}{1}
\section{Introduction}\label{sec.intro}

The aim of this note is to develop some new techniques regarding the application of symmetrization methods to Dirichlet fractional elliptic problems of the type
\begin{equation} \label{eq.0}
\left\{
\begin{array}
[c]{lll}%
\left( -\Delta\right)^{s}u=f & & \text{in }%
\Omega,\\
\\
u=0 & & \text{on }\R^{N}\setminus\Omega,
\end{array}
\right. %
\end{equation}
where
$\Omega\subset\R^{N}$ ($N\geq1$) is a smooth bounded open set, the source term $f=f(x)$ is assumed to belong to $L^{p}(\Omega)$ for suitable $p\geq1$ and $s\in (0,1)$.
In the case of local equations, it is well known that under Schwarz symmetrization
the solution to the homogeneous Dirichlet problem for an elliptic equation increases
in terms of rearrangements. Namely,
consider
the Dirichlet problem
\begin{equation} \label{eq.local}
\left\{
\begin{array}
[c]{lll}%
-\big(a_{ij}\;z_{x_i}\big)_{x_j}=f & & \text{in }%
\Omega,\\
\\
z=0 & & \text{on }\partial\Omega,
\end{array}
\right. %
\end{equation}
where the measurable coefficients $a_{ij}=a_{ij}(x)$ satisfy the ellipticity condition
\begin{equation*}
a_{ij}(x)\xi_i\xi_j\ge|\xi|^2,\qquad \forall\xi\in\R^N,\text{ a.e. }x\in\Omega.
\end{equation*}
A nowadays classical result (see, for instance, \cite{wein}, \cite{maz}, \cite{Talenti1}) states that
if $z\in H_0^1(\Omega)$ is the weak solution to \eqref{eq.local} and $w\in H_0^1(\Omega^\#)$ is the weak solution to the ``symmetrized problem''
\begin{equation*} 
\left\{
\begin{array}
[c]{lll}%
-\Delta w=f^\# & & \text{in }%
\Omega^\#,\\
\\
w=0 & & \text{on }\partial\Omega^\#,
\end{array}
\right. %
\end{equation*}
then
\begin{equation}\label{comp}
z^\#(x)\le w(x),\qquad x\in\Omega^\#.
\end{equation}
Here $\Omega^\#$ is the ball centered at the origin such that $|\Omega^\#|=|\Omega|$ and $z^\#$ denotes the
 Schwarz symmetrization of $z$ (see Section \ref{Sec2} for further details):
\begin{equation*}
 z^\#(x)=\sup\{t\ge0:|\{x:|z(x)|>t\}|>\omega_N|x|^N\},
\end{equation*}
where $\omega_N$ is the measure of the unit ball in $\R^N$.
An immediate consequence of inequality \eqref{comp} is, for example, that any norm of $z$ increases under Schwarz symmetrization.

The approach used in most of the papers concerning symmetrization techniques is based on the fact that the use of a suitable test function allows to obtain, for a.e. $t\in(0,\sup u)$, the inequality
\begin{equation}\label{first}
-\frac d{dt}\int_{|z|>t}|Dz|^2dx\le \int_{z^\#>t}f^\#(x)\,dx.
\end{equation}
Schwarz inequality, Fleming-Rishel formula and isoperimetric inequality are then used in order to obtain a first order differential inequality involving $z^\#$ and its radial derivative. Finally, a comparison principle gives \eqref{comp}. A slightly different approach has been used in \cite{lions}, where the author observes that in inequality \eqref{first} one can use the so-called P\'olya-Szeg\"o principle which states that, if $u\in H_0^1(\Omega)$, then
\begin{equation}\label{polya}
\int_{\Omega}|Du|^2dx\ge \int_{\Omega}|Du^\#|^2dx.
\end{equation}
Actually, the differential quotient used to compute the derivative in \eqref{first} can be written in terms of the Dirichlet integral of a suitable truncation of $z$, which is a Sobolev function, so \eqref{polya} applies to give
\begin{equation}\label{second}
-\frac d{dt}\int_{z^\#>t}|Dz^\#|^2dx\le \int_{z^\#>t}f^\#(x)\,dx.
\end{equation}
At this point the integral on the left hand side concerns a radially symmetric function and the quoted first order differential inequality involving $z^\#$ follows immediately, without the use of isoperimetric inequality.

The literature about the possible extensions of \eqref{comp} is wide and, confining ourselves only to the case of homogeneous Dirichlet conditions, we recall symmetrization results for elliptic equations with lower order terms (\cite{ATlincei}), for $p$-Laplacian type equations (\cite{talNON}), for porous medium equation (\cite{Vsym82}), for parabolic equations (\cite{Band2}), for anisotropic equations (\cite{aflt}, \cite{cianchia}).

Actually, the effect of symmetrization on fractional elliptic problems like \eqref{eq.0} has already been exploited in \cite{dBVol} and then in \cite{VazVol1}, \cite{VazVol2}, \cite{VazVolSire}, \cite{VOLZNONLINEAR}, \cite{feostingaV}. In those papers a symmetrization result in terms of mass concentration (\emph{i.e.}, an integral comparison, as in the parabolic case) is obtained in a somewhat indirect way. Indeed,
%
it has been used in an essential way the fact that problem \eqref{eq.0} can be linked to a suitable, local \emph{extension problem}, whose solution $\psi(x,y)$, an \emph{harmonic extension} of $u$, is defined
on the infinite cylinder $\mathcal{C}_{\Omega}=\Omega\times(0,\infty)$, to which classical symmetrization techniques (with respect to the variable $x\in\Omega$)
can be applied: the difficulties in this approach is the translation of the boundary conditions and the presence of the ``extra'' variable $y\geq0$,
which is fixed in the symmetrization arguments, an important detail which allows to use a Steiner symmetrization approach. Then an integral (or mass
concentration) comparison is naturally expected, and, being $u$ the trace of $\psi$ over $\Omega\times\left\{0\right\}$, the comparison result for the
extension $\psi$ of $u$ immediately implies an estimate for $u$. Furthermore, an absolutely non negligible aspect characterizing the works \cite{dBVol}, \cite{VazVol1}, \cite{VazVol2}, \cite{VOLZNONLINEAR} from \cite{VazVolSire} is the fact that the former ones deal with the \emph{spectral} version of the fractional Laplacian operator $(-\Delta)^{s}_{spec}$ on $\Omega$ (which is defined on a domain encoding de facto the boundary conditions), while the latter one considers the so called \emph{restricted} version of the fractional Laplacian. In any case, all these approaches takes benefit from the \emph{local} intepretation of the fractional Laplacian as the Dirichlet-to-Neumann map, \emph{i.e.}, as an outward normal derivative on the boundary of the half space $H=\left\{y>0\right\}$ of the solution $\psi$ (the so called $s$-harmonic extension) of a local extension problem posed on $H$, being $H$ replaced by $\mathcal{C}_{\Omega}$ (with suitable boundary condition on the lateral surface): this is essentially the nowadays classical result by Caffarelli and Silvestre \cite{Caffarelli-Silvestre}, generalized in \cite{Stinga-Torrea}.

Our aim is now \emph{not} to use the local interpretation of the fractional Laplacian in the derivation of the mass concentration comparison for the solution to problem \eqref{eq.0} in terms of the solution $v$ to the symmetrized problem
\begin{equation}
\left\{
\begin{array}
[c]{lll}%
\left( -\Delta\right) ^{s}v=f^{\#} & & \text{in }%
\Omega^{\#}\\
\\
v=0 & & \text{on }\R^{N}\setminus\Omega^{\#}.
\end{array}
\right. \label{eq.1}%
\end{equation}
The main goal is then a new proof of the mass concentration comparison
\begin{equation}
u^\#(x)\prec v(x)\label{concinequlin}
\end{equation}
where the above comparison (see Section \ref{Sec2} for more details) means that, for every $r>0$, it holds
\[
\int_{|x|<r}u^\#(x)\,dx\leq \int_{|x|<r}v(x)\,dx.
\]
Even though comparison \eqref{concinequlin} has been already obtained with the techniques described above, however we would like to point out that in our opinion the results contained in the present paper could be of particular interest because the arguments used to prove \eqref{concinequlin} are completely new and they seem to be very flexible with respect to those used in previous papers. Furthermore, we include an observation about optimality of \eqref{concinequlin}, which we have not found elsewhere.

As regards the novelty of our approach, we observe that we develop techniques which are in some sense \emph{intrinsic}, that is, we use directly the weak formulation of solution to problem \eqref{eq.0} without using any \emph{local} extension. The main original steps in the proof are two.

In the first step, inspired by \cite{lions}, we use the nonlocal version of the P\'olya-Szeg\"o principle which holds true in fractional Sobolev spaces.
We are able to show that it is possible to apply such a principle to an integral containing the solution $u$ and a truncated of $u$ in order to obtain a new inequality which can be seen as the nonlocal counterpart of  inequality
    \eqref{first}. We are then reduced to consider an inequality where the solution $u^\#$ is already rearranged, but it appears on the left-hand side a quantity which appears to be, roughly speaking, a kind of mass concentration of the $s$-Laplacian of $u^\#$. However, such an interpretation cannot be completely justified because $u^\#$ lacks the required regularity. So, in the second relevant step, we are able to rewrite the obtained inequality as a differential inequality that involves the $s$-Laplacian of the mass concentration of $u^\#$ computed on $\R^{N+2}$. Thus, in some sense, comparison \eqref{concinequlin} becomes quite natural.\\[8pt]
\indent It is worth to spend some words concerning the flexibility of our approach and its several advantages. First, we point out that it definitely clarifies a certain continuity of the comparison result with respect to the parameter $s\in(0,1)$, in the sense that Talenti's pointwise result is recovered in the limit as $s\rightarrow 1$ (which looks clearer in Figures 1 and 2 of Section \ref{counter}): this remark cannot be achieved using the extension method techniques employed in the previous works on the subject. On the other hand, we observe that our approach could be used in various contexts. As a matter of fact, because of the fact that P\'olya-Szeg\"o principle holds true in more general situations, the extension to various  classes of nonlocal PDEs seems to be possible. For example, our methods appear to be suitable for the investigation about the effects of symmetrization in cases where, apparently, a corresponding approach via an extension problem is not available. Possible examples in the elliptic framework are nonlocal semilinear equations or equations involving elliptic integro-differential operators with general kernels  of the L\'evy type, \emph{e.g.} operators in the form
\[
L_{K}u=\;\text{P.V.} \int_{\R^N}\big[(u(x)-u(y)\big]K(x,y)dy
\]
where $K$ is a symmetric, possibly singular, nonnegative kernel satisfying
\[
\int_{\R^{N}}\min\left\{1,|y|^{2}\right\}K(y)dy<\infty.
\]
Such operators are widely studied in literature, see, \emph{e.g.}, \cite{ROSOTONSURVEY} and the extensive literature therein. Another possibility would be in trying to adapt our methods to nonlinear equations involving the so-called fractional $p$-Laplacian operator, \emph{i.e.}, the nonlocal nonlinear operator defined for $1<p<\infty$ (see \emph{e.g.} \cite{VazquezFPL}, \cite{BrascoLindSchi}, \cite{Iannizzotto})
\[
(-\Delta)_{p}^su=\text{P.V.} \int_{\R^N}\frac{\Phi(u(x)-u(y))}{|x-y|^{N+sp}}dy,
\]
where $\Phi(z):=|z|^{p-2}z$, $z\in\R$.  In all the above-cited examples, \emph{no} extension technique is possible to reduce to local interpretations.
On the other side, our elliptic methods could be employed for deriving mass concentration comparison for the \emph{parabolic evolution} equations (linear and nonlinear, in bounded or unbounded domains) with the diffusion terms given by one of the above cited nonlocal operators. It is very-well known that some applications of such symmetrization results for parabolic equations are, for instance, the rather immediate derivation of time decay estimates with sharp constants when the qualitative properties of the selfsimilar fundamental solutions are known, see for instance \cite{VazVol1}, \cite{VazVol2}. For a consistent survey of the important applications of mass concentration comparison results in the field of nonlinear parabolic equations, see \cite{Vazquez2007}. Finally, an interesting point would be to push forward the applications of our techniques to the hot topic of the theory of \emph{aggregation diffusion equations}, in which symmetrization can be a powerful tools in characterizing the geometry of the asymptotic profile, see \emph{e.g.} \cite{CHVY}. We plan to address these topics in forthcoming papers.\\[8pt]
\indent As regards the optimality of \eqref{concinequlin}, the fact that it is possible to use in the nonlocal context an approach similar to the one used in the local case could indicate that a pointwise estimate as \eqref{comp} could be true also for problem \eqref{eq.0}. We are able to exhibit, for any $s\in(0,1)$, some counterexamples which show that \eqref{comp} does not hold in general.

We finally observe that, even though \eqref{concinequlin} is weaker with respect to a pointwise estimate as \eqref{comp}, however it implies that any norm of $u$ increases under Schwarz symmetrization. As a consequence, we get optimal estimates of the $L^p$ norms of $u$
and we are also able to prove a comparison between the nonlocal energy of $u$ and $v$.

The paper is organized as follows. In Section \ref{Sec2} some preliminary results and notation are collected. Section \ref{Sec3} contains the main comparison result, some applications and remarks. In Section \ref{counter} we discuss some counterexamples, while in Section \ref{Sec4} we prove the main theorem, splitting the proof in several steps. Finally, in Section \ref{Sec5} some possible extensions are discussed, together with a few remarks.

%

\section{Preliminaries and notation} \label{Sec2}

For the proof of the main results we need some preliminary results concerning symmetrization, functional spaces, Fourier representation and hypergeometric functions. So, in this section we give a brief account of such properties and we fix the notation used in the sequel.

\subsection{Rearrangements and symmetrization}

We briefly recall the basic notions of Schwarz symmetrization and some related fundamental properties. Readers who are interested in more details
of the theory are warmly addressed to the classical monographs \cite{Hardy}, \cite{Bennett},  \cite{Kesavan}, \cite{Bandle} or to the paper
\cite{Talentirearrinv}.

A measurable real function $f$ defined on $\R^{N}$ is called \emph{radially symmetric} (or \emph{radial}) if there is a function
$\widetilde{f}:[0,\infty)\rightarrow \R$ such that $f(x)=\widetilde{f}(|x|)$ for all $x\in \R^{N}$. We will often write $f(x)=f(r)$,
$r=|x|\ge0$ for such functions by abuse of notation. We say that $f$ is \emph{rearranged} if it is radial, nonnegative and $\widetilde{f}$ is a
right-continuous, non-increasing function of $r>0$. A similar definition can be applied for real functions defined on a ball
$B_{R}(0)=\left\{x\in\R^{N}:|x|<R\right\}$.


Let $f$ be a real measurable function on $\R^N$. If $f$ is such that
its
\emph{distribution function} $\mu_{f}$ satisfies%
\begin{equation}\label{distribution}
\mu_{f}(  t)  :=\left\vert \left\{  x\in\Omega:\left\vert f\left(
x\right)  \right\vert >t\right\} \right\vert<+\infty, \qquad\text{for every }t
>0,
\end{equation}
we define the \emph{one dimensional decreasing rearrangement} of $f$ as%
\[
f^{\ast}\left(  \sigma\right)  =\sup\left\{ t\geq0:\mu_{f}\left(  t\right)
>\sigma\right\}  \text{ , }\sigma\in\left(  0,\left\vert \Omega\right\vert \right).
\]
If $f$ is a real measurable function on an open set $\Omega\subset\R^N$ we extend $f$  as the  zero function in $\R^N\backslash\Omega$ and we define the one dimensional decreasing rearrangement of $f$ as the rearrangement of such an extension. This means that $f^{\ast}(\sigma)=0$ for $\sigma\in[|\Omega|,\infty)$. From the above definition it follows
that $\mu_{f^{\ast}}=\mu_{f}$ (\emph{i.e.,} $f$ and $f^{\ast}$ are equi-distributed) and $f^{\ast}$ is exactly  the \emph{generalized right
inverse function} of
$\mu_{f}$.
Furthermore, if $\Omega^{\#}$ is the ball of $%
\mathbb{R}
^{N}$ centered at the origin having the same Lebesgue measure as $\Omega$ ($\Omega^{\#}=\R^N$ if $|\Omega|=+\infty$), we define the
function
\[
f^{\#}\left(  x\right)  =f^{\ast}(\omega_{N}\left\vert x\right\vert
^{N})\text{ \ , }x\in\Omega^{\#},
\]
that will be called \emph{radially decreasing rearrangement}, or \emph{Schwarz decreasing rearrangement}, of $f$. We easily infer that $f$ is rearranged if and
only if
$f=f^{\#}$.

A simple  consequence of the definition is that rearrangements preserve
$L^{p}$
norms, that is, for all $p\in[1,\infty]$
\[
\|f\|_{L^{p}(\Omega)}=\|f^{\ast}\|_{L^{p}(0,|\Omega|)}=\|f^{\#}\|_{L^{p}(\Omega^{\#})}\,;
\]
furthermore, the classical Hardy-Littlewood inequality holds true
\begin{equation}
\int_{\Omega}\vert f(x)\,  g(x)  \vert dx\leq\int_{0}^{\left\vert \Omega\right\vert}f^{\ast}(\sigma)\,  g^{\ast}(\sigma)  d\sigma=\int_{\Omega^{\#}}f^{\#}(x)\,g^{\#}(x)\,dx\,,
\label{HardyLit}%
\end{equation}
where $f,g$ are measurable functions on $\Omega$.

Here we recall an important ingredient in the proof of our main result, corresponding to the following generalization of the \emph{Riesz rearrangement inequality} (see \cite[Theorem 2.2]{ALIEB}).
\begin{theorem}
Let $F:\R^{+}\times\R^{+}\rightarrow\R^{+}$ be a continuous function such that $F(0,0)=0$ and
\begin{equation}
F(u_{2},v_{2})+F(u_{1},v_{1})\geq F(u_{2},v_{1})+F(u_{1},v_{2})\label{F}
\end{equation}
whenever $u_{2}\geq u_{1}>0$ and $v_{2}\geq v_{1}>0$.
Assume that $f, g$ are nonnegative measurable functions on $\R^{N}$ which satisfy \eqref{distribution}, then we have the inequalities
\begin{equation}
\int_{\R^{N}}\int_{\R^{N}}F(f(x),g(y))W(ax+by)\,dx\,dy\leq \int_{\R^{N}}\int_{\R^{N}}F(f^{\#}(x),g^{\#}(y))W(ax+by)\,dx\,dy\label{mainRieszineq}
\end{equation}
and
\[
\int_{\R^{N}}F(f(x),g(x))\,dx\leq \int_{\R^{N}}F(f^{\#}(x),g^{\#}(x))\,dx,
\]
for any nonnegative function $W\in L^{1}(\R^{N})$ and any choice of nonzero numbers $a$ and $b$.
\end{theorem}


\subsection{Mass concentration}

Since we will provide estimates of the solutions of our fractional elliptic problem in terms of their integrals, the following
definition (see, for instance, \cite{ChRice}, \cite{ALTa}, \cite{Vsym82}) is of basic importance.

\begin{definition}
Let $f,g\in L^{1}_{loc}(\R^{N})$.
We say that $f$ is less concentrated than $g$, and we write
$f\prec
g$ if for
all $r>0$ we get
\[
\int_{B_{r}(0)}f^\#(x)\,dx\leq \int_{B_{r}(0)}g^\#(x)\,dx.
\]
\end{definition}
The partial order relationship $\prec$ is called \emph{comparison of mass concentrations}.
Of course, this definition can be suitably adapted if $f,g$ are defined in an open set $\Omega$ (considering the extension to zero outside $\Omega$). Moreover, we have that
$f\prec g$ if and only if
\[
\int_{0}^{\sigma}f^{\ast}(\tau)\,d\tau\leq \int_{0}^{\sigma}g^{\ast}(\tau)\,d\tau,
\]
for all $\sigma\geq0$.

The comparison of mass concentrations enjoys some nice equivalent formulations
(for the proof we refer to
\cite{Chong}, \cite{ALTa}, \cite{VANS05}).

\begin{lemma}\label{lemma1}
Let $f,g\in L_+^{1}(\Omega)$. Then the following are equivalent:

\vskip7pt
\noindent(i) $f\prec g$;
\vskip7pt

\noindent(ii) for all $\phi\in L^\infty_+(\Omega)$,
$$\int_{\Omega}f(x)\phi(x)\,dx\leq \int_{\Omega^\#}f^\#(x)\phi^\#(x)\,dx.
$$
\vskip7pt

\noindent(iii) for all convex, nonnegative functions $\Phi:[0,\infty)\rightarrow [0,\infty)$ with $\Phi(0) = 0$ it holds
$$\int_{\Omega}\Phi(f(x))\,dx\leq \int_{\Omega}\Phi(g(x))\,dx.
$$
%
\end{lemma}
We explicitly observe that, if $f, g\in L^p(\Omega)$ $(1 < p \le\infty)$, then we may take $\phi \in L^{p'}(\Omega)$ in the point (ii) above.

From this Lemma it easily follows that if $f\prec g$,
then
\begin{equation}
\|f\|_{L^{p}(\Omega)}\leq \|g\|_{L^{p}(\Omega)}\quad \forall p\in[1,\infty].
\end{equation}

{
\subsection{Functional spaces and some computations for radial functions}

It is well known that for $s\in(0,1)$ the fractional Laplacian of a smooth real function $u$ on $\R^N$ can be equivalently defined as a pseudodifferential operator by means of
$$(-\Delta)^su=\mathcal{F}^{-1}\big(|\xi|^{2s}\mathcal{F}(u)(\xi)\big).
$$
where $\mathcal{F}$ is the Fourier transform, and in terms of a hypersingular integral
\[
(-\Delta)^su(x)=\gamma(N,s)\;\text{P.V.} \int_{\R^N}\frac{u(x)-u(y)}{|x-y|^{N+2s}}dy,
\]
where the explicit value of the normalization constant $\gamma(N,s)$ is given by
\begin{equation}
\gamma(N,s)=\frac{s2^{2s}\Gamma\big(\frac{N+2s}2\big)}{\pi^{\frac N2}\Gamma(1-s)}.\label{costante}
\end{equation}

In general, we can define $(-\Delta)^su$ in the distributional sense when $u$ has a strong enough decay at infinity, \emph{e.g.}, when $u$ belongs to the weighted space (see for instance \cite{Silvestre-thesis})
$$L_s(\R^N)=\left\{u:\R^N \rightarrow\R\text{ such that } \int_{\R^N} \frac{|u(x)|}{1+|x|^{N+2s}} dx<+\infty\right\}.
$$
If $\Omega$ is an open set of $\R^N$ and $s\in (0,1)$ we introduce the fractional Sobolev space $H^{s}(\Omega)$, defined as
\[
H^{s}(\Omega)=\left\{u\in L^{2}(\Omega):\,[u]_{H^{s}(\Omega)}<\infty\right\},
\]
where $[\cdot]_{H^{s}(\Omega)}$ is the Gagliardo seminorm
\[
[u]_{H^{s}(\Omega)}=\left(\int_{\Omega}\int_{\Omega}\frac{|u(x)-u(y)|^{2}}{|x-y|^{N+2s}}dx\,dy\right)^{1/2}.
\]
We have that $H^{s}(\Omega)$ is a Hilbert space w.r. to the scalar product
\[
(u,v)_{H^{s}(\Omega)}=(u,v)_{L^{2}(\Omega)}+\left(\int_{\Omega}\int_{\Omega}\frac{(u(x)-u(y))(v(x)-v(y))}{|x-y|^{N+2s}}dx\,dy\right)^{1/2}.
\]
When $\Omega=\R^{N}$ one can prove that (see for instance \cite{hitch}) $H^{s}(\R^N)=\widehat H^{s}(\R^N)$, where
$$\widehat H^{s}(\R^N)=\left\{u\in L^2(\R^N):\int_{\R^N}(1+|\xi|^{2s})|\hat u(\xi)|^2d\xi<+\infty\right\},$$
so that
$$\|u\|_{H^s(\R^N)}=\|u\|_{L^2}+\|(-\Delta)^{s/2}u\|_{L^2}.
$$
We now introduce the fractional Sobolev spaces where weak solutions to to problems of the type \eqref{eq.0} are naturally settled.
For a bounded Lipschitz domain $\Omega$ and $s\in (0,1)$ we denote by $H_{0}^{s}(\Omega)$ the closure of $C_{0}^{\infty}(\Omega)$ w.r. to the $H^{s}(\Omega)$ norm, furthermore we define the interpolation space
$$H_{00}^{1/2}(\Omega)=\left\{u\in H^{1/2}(\Omega):\int_{\Omega}\frac{u^{2}(x)}{d^2(x)}dx<\infty\right\}$$ with $d(x)=\text{dist}(x,\partial\Omega)$.\\
Since the fractional Laplacian $(-\Delta)^{s}$ will be evaluated on functions $u$ compactly supported in $\overline{\Omega}$ (because of the homogeneous boundary conditions), the domain of the $(-\Delta)^{s}$ (which is often called \emph{restricted fractional Laplacian} on $\Omega$) will be
\[
{\cal H}^{s}(\Omega)=\left\{u\in H^{s}(\R^{N}):\text{supp}( u)\subset\overline{\Omega}\right\}
\]
which can be identified as follows (see \cite{SirBonfVaz}):
\begin{equation}
{\cal H}^{s}(\Omega)=
\left\{
\begin{array}
[c]{lll}%
H^{s}(\Omega) &  & if\,0<s<1/2,\\[6pt]
H_{00}^{1/2}(\Omega) &  & if\,s=1/2\\[6pt]
H_{0}^{s}(\Omega) &  & if\,1/2<s<1.
\end{array}
\right.  \label{spaceH}%
\end{equation}
Now we recall some results concerning the representation of Fourier transform and of fractional Laplacian applied to a radial function.
The following result can be found in \cite[Theorem 40 and  Ch. IV, $\mathsection$ 5]{BC}.

\begin{theorem} [Fourier-Bessel representation] \label{bessel}
Let $u(x) = u(|x|)$ be a radial function, and suppose that
$$\tau \rightarrow \tau^{N} u(\tau)J_{\frac N2-1} (\tau) \in L^{1}(\R^{+}),$$
where $J_{\frac N2-1}(t)$ denotes the Bessel function of order $\frac N2-1$. Then, the Fourier transform of $u$ is given by
$$\mathcal{F}(u)(\xi)=2\pi|\xi|^{-\frac N2+1}\int_{0}^{+\infty}\tau^{\frac N2}u(\tau)J_{\frac N2-1} (2\pi|\xi|\tau)\,d\tau.$$
Furthermore, if $u\in L^2(\R^N)$, then formula above remains valid in $L^2(\R^N)$.
\end{theorem}


The following result gives the expression of the fractional Laplacian in radial coordinates and it can be found in \cite{garo}.

\begin{theorem} \label{bessela}
Let $u(x) = u(|x|)$ be a radial function. Then
$$(-\Delta)^{s}u(x) = \frac{(2\pi)^{2s+2}}{|x|^{\frac N2-1}}\int_{0}^{+\infty}
\rho^{1+2s}J_{\frac N2-1} (2\pi|x|\rho)\left(\int_{0}^{+\infty}\sigma^{\frac N2}u(\sigma)J_{\frac N2-1} (2\pi\rho\sigma)\,d\sigma\right)d\rho
$$
provided that the integrals exist and are convergent.
\end{theorem}

}

\subsection{Some properties of hypergeometric functions}\label{Hyperg}

We now recall a few properties of the hypergeometric function $_{2}F_1(a,b;c;x)$ (see, for example, \cite[Ch. II]{magnus}), which, for $c>b>0$ and $0<\tau<1$, can be represented as
\begin{equation}\label{represent}
_{2}F_1(a,b;c;x)=\frac{\Gamma(c)}{\Gamma(b)\Gamma(c-b)}\int_0^1\tau^{b-1}(1-\tau)^{c-b-1}(1-x\tau)^{-a}d\tau
\end{equation}
These properties will turn out useful in the proof of the main Theorem, namely Theorem \ref{main}.
Classical results about the derivatives of $_{2}F_1(a,b;c;x)$ read as
\begin{align*}\allowdisplaybreaks
_{2}F_1'(a,b;c;x)&=\frac {ab}c {\ }_{2}F_1(a+1,b+1;c+1;x),\displaybreak[1]\\ \\
_{2}F_1(a+1,b;c+1;x)&=\frac c{c-b} {\ }_{2}F_1(a,b;c;x)-\frac c{c-b}\frac{1-x}a{\ }_{2}F_1'(a,b;c;x),\displaybreak[1]\\ \\
_{2}F_1(a-1,b;c-1;x)&=\frac {c-1-bx}{c-1}{\ }_{2}F_1(a,b;c;x)+\frac{x(1-x)}{c-1}{\ }_{2}F_1'(a,b;c;x).
\end{align*}
A straightforward consequence of the above equalities is the following one:
\begin{equation}\label{formulona}
_{2}F_1'(a,b;c;x)=\frac {ab}{c}{\ }_{2}F_1(a+1,b;c+1;x)+\frac {ax}{c}{\ }_{2}F'_1(a+1,b;c+1;x).
\end{equation}
We also recall the following equality which holds true for $b>0$ and $|x|<1$.
\begin{equation}\label{hyper}
\int_{0}^{\pi}\frac{\sin^{2b-1}\theta}{(1-2x\cos\theta+x^{2})^{a}}d\theta=\frac{\sqrt\pi\;\Gamma(b)}{\Gamma(b+\frac12)}{\ }_{2}F_1(a,a-b+\tfrac12;b+\tfrac12;x^{2}).
\end{equation}

We finally recall that a direct computation in \eqref{represent} gives
\begin{equation}\label{zero}
{}_{2}F_{1}(a,b;c;0)=1,
\end{equation}
and, when $c>a+b$ a formula due to Gauss states
\begin{equation}\label{gauss}
{}_{2}F_{1}(a,b;c;1)={\frac {\Gamma (c)\Gamma (c-a-b)}{\Gamma (c-a)\Gamma (c-b)}}.
\end{equation}
The above information, together with the equality
\begin{equation}\label{linear}
{}_{2}F_{1}(a,b;c;x)=(1-x)^{c-a-b}
{}_{2}F_{1}(c-a,c-b;c;x),
\end{equation}
will be useful in order to establish some asymptotic behaviours.

{
\section{Main results and remarks}\label{Sec3}
We are now in position to give the
following definition of weak solution to problems of the type \eqref{eq.0} (see \cite{ROSOTONSURVEY})
\begin{definition}\label{weaksol}
Let $f\in L^{p}(\Omega)$, for some $p\geq 2N/(N+2s)$ when $N\geq2$ and any $p>1$ for $N=1$. A weak solution to problem \eqref{eq.0} is a function $u\in{\cal H}^{s}(\Omega)$ such that the equality
\begin{equation}
(u,\varphi)_{{\cal H}^s(\Omega)}:=\frac{\gamma(N,s)}{2}\int_{\R^{N}}\int_{\R^{N}}\frac{\left(u(x)-u(y)\right)\left(\varphi(x)-\varphi(y))\right)}{|x-y|^{N+2s}}dx\,dy=
\int_{\Omega}f(x)\,\varphi(x)\,dx\label{weakform}
\end{equation}
holds for all test functions $\varphi\in {\cal H}^{s}(\Omega)$ .
\end{definition}
It is clear that the bilinear form $(\cdot,\cdot)_{{\cal H}^{s}(\Omega)}$ verifies the classical Lax-Milgram Theorem, thus a unique weak solution $u$ to problem \eqref{eq.0} exists because $f\in{\cal H}^{-s}(\Omega)=:({\cal H}^{s}(\Omega))^{\prime}$ due to the fractional Sobolev embeddings.
As regards important properties of solutions, such as maximum principles, regularity results or extensions to a wider class of operators we refer to \cite{ROSOTONSURVEY}. 
We also recall that the solution $v$ to the symmetrized problem \eqref{eq.1} is radially strictly decreasing, see for instance \cite[Theorem 1.1]{felmw} or \cite[Theorem 1.1]{Barrios}.

Now we can finally state the main result of this paper.

\begin{theorem}\label{main}
Let $s\in(0,1)$ and let $f\in L^{p}(\Omega)$, with $p\geq 2N/(N+2s)$ when $N\geq2$ and any $p>1$ for $N=1$. If $u$ is the weak solution to  problem $\eqref{eq.0}$ and $v$ is
the solution to the corresponding symmetrized problem \eqref{eq.1}, we have
\begin{equation}
u^\#(x)\prec v(x).\label{compariso}
\end{equation}
\end{theorem}

Theorem \ref{main} has a certain number of interesting implications. For instance, \eqref{compariso} and property \eqref{lemma1} transfers the study of the $L^p$ regularity scale of the solution $u$ to \eqref{eq.0} to the same regularity for the solution $v$ to the radial problem \eqref{eq.1}.  The advantage of this step relies in the fact that $v$ can be written in the integral form in terms of the Green function on the ball, which is explicit. We can summarize all these considerations in the following result.
\begin{theorem}\label{regularity}
Let $N\geq2$, $f\in L^{p}(\Omega)$, with $p\geq 2N/(N+2s)$, and let $u$ be the weak solution to  problem $\eqref{eq.0}$. We have:
\begin{enumerate}
\item if $p<N/(2s)$ then $u\in L^{q}(\Omega)$, with
\begin{equation}
q=\frac{Np}{N-2sp}\label{q}
\end{equation}
and there exists a constant $C$ such that:
\[
\|u\|_{L^{q}(\Omega)}\leq C\|f\|_{L^{p}(\Omega)};
\]
\item if $p>N/(2s)$ then $u\in L^{\infty}(\Omega)$ and there exists a constant $C$ such that:
\[
\|u\|_{L^{\infty}(\Omega)}\leq C\|f\|_{L^{p}(\Omega)};
\]
\item if $p=N/(2s)$, then $u\in L_{\Phi_{p}(\Omega)}$ and there exists a constant $C$ such that:
\[
\|u\|_{L_{\Phi_{p}(\Omega)}}\leq C\|f\|_{L^{p}(\Omega)},
\]
where $L_{\Phi_{p}(\Omega)}$ is the Orlicz space generated by the $N$-function
\[
\Phi_{p}(t)=\exp(|t|^{p^{\prime}})-1.
\]
\end{enumerate}
\end{theorem}
\begin{proof}
The proof follows some arguments of \cite[Theorems 4.3-4.4]{dBVol}, but we propose here the details for the sake of completeness.
We write
\[
v(x)=\int_{\Omega^{\#}}\mathsf{G}_{\Omega^{\#}}(x,y)f^{\#}(y)dy,
\]
where $\mathsf{G}_{\Omega^{\#}}$ is the Green function of the restricted fractional Laplacian on the ball (see \cite{BUCUR}). Since (see for instance \cite{Kul})
\begin{equation}
\mathsf{G}_{\Omega^{\#}}(x,y)\leq \frac{C}{|x-y|^{N-2s}}\label{Greenfuncest}
\end{equation}
for any $x\neq y$ in $\Omega^{\#}$. Then extending $f$ to 0 out of $\Omega$, Hardy-Littlewood-Sobolev inequality (see \cite{lloss}) implies, for $p<N/(2s)$,
\[
\|u\|_{L^{q}}\leq \|v\|_{L^{q}}\leq 
C \|f\|_{L^{p}}
\]
where $q$ is given by \eqref{q}. The case $p>(N/2s)$ is even easier, because
\begin{align*}
\|u\|_{L^{\infty}(\Omega)}&\leq \|v\|_{L^{\infty}(\Omega)}\leq C\int_{\Omega^{\#}}\frac{f^{\#}(y)}{|y|^{N-2s}}dy\\
&\leq C \|f\|_{L^{p}(\Omega)}\left(\int_{\Omega^{\#}}\frac{1}{|y|^{(N-2s)p^{\prime}}}dy\right)^{1/p^{\prime}}<\infty.
\end{align*}
The limit case is a bit more elaborate. Indeed, it can be proven that (see for instance \cite[Lemma 6.12]{Bennett})  $L_{\Phi_{p}(\Omega)}$ can be interpreted as the space of all measurable functions $\psi$ such that
\[
\sup_{\sigma\in(0,|\Omega|)}\frac{\psi^{\ast\ast}(\sigma)}{\left(1+\log\left(\frac{|\Omega|}{\sigma}\right)\right)^{1/p^{\prime}}}<\infty,
\]
where
\[
\psi^{\ast\ast}(\sigma)=\frac{1}{\sigma}\int_{0}^{\sigma}\psi^{\ast}(\tau)d\tau.
\]
Since (see \cite[Lemma 1.6]{oneil1963})
\[
(f^{\#}\ast|\cdot|^{2s-N})^{\ast\ast}(\sigma)\leq \int_{s}^{|\Omega|}\tau^{(2s/N)-N}f^{\ast\ast}(\tau)d\tau,
\]
and the fact that the $L^p$ norms of $f^{\ast}$ and $f^{\ast\ast}$ are equivalent (see \cite[Lemma 4.5]{Bennett})
an easy application of H\"older inequality and \eqref{compariso} provides
\begin{align*}
\frac{u^{\ast\ast}(\sigma)}{\left(1+\log(\frac{|\Omega|}{\sigma})\right)^{1/p^{\prime}}}&\leq C
\frac{(f^{\#}\ast|\cdot|^{2s-N})^{\ast\ast}(\sigma)}{\left(1+\log(\frac{|\Omega|}{\sigma})\right)^{1/p^{\prime}}}\\
&\leq C\frac{\left(\log\left(\frac{|\Omega|}{\sigma}\right)\right)^{1/p^{\prime}}}{\left(1+\log(\frac{|\Omega|}{\sigma})\right)^{1/p^{\prime}}}\|f\|_{L^{p}(\Omega)},
\end{align*}
which concludes the proof.
\end{proof}
\begin{remark}
For the sake of simplicity in the above theorem we have supposed $N\ge2$,
but similar arguments can be used also in case $N=1$. In such a case, if
$s<1/2$ estimate \eqref{Greenfuncest} still holds, thus Theorem \ref{regularity} remains true. If $s>1/2$ the Green function on a symmetric interval is bounded \cite[Corollary 3]{Bycz}, while in the special case $s=1/2$ the Green function is explicit (see Section \ref{counter}) and one has
\[
\mathsf{G}_{\Omega^{\#}}(x,y)\leq C\log\frac{1}{|x-y|}
\]
where we recall that $-(1/\pi)\log|x|$ is exactly the fundamental solution of $(-\Delta)^{1/2}$ for $N=1$. Thus for $s\leq1/2$ the solution $u$ is bounded for all $p>1$.
\end{remark}
It can be shown that optimal embedding results hold when the Lorentz spaces $L^{p,q}(\Omega)$ are introduced, see for such questions \cite{dBVol} and \cite{VOLZNONLINEAR}.\\[2pt]
\noindent Another interesting consequence of \eqref{compariso} is the estimate of the nonlocal energy of $u$ in terms of the one of $v$, in the spirit of \cite{Talenti1}.
\begin{proposition}\label{energies}
Under the assumptions of Theorem \eqref{main}, we have
\begin{equation}
[u]_{H^{s}(\R^{N})}\leq [v]_{H^{s}(\R^N)}.\label{energineq}
\end{equation}
\end{proposition}
\begin{proof}
Simply inserting $u$ as test function in the weak formulation \eqref{weakform}, an employ of Hardy-Littlewood rearrangement inequality and Lemma \ref{lemma1} provides
\begin{align*}
\frac{\gamma(N,s)}{2}\int_{\R^{N}}\int_{\R^{N}}\frac{|u(x)-u(y)|^{2}}{|x-y|^{N+2s}}dx\,dy&=
\int_{\Omega}f(x)\,u(x)\,dx\\
&\leq \int_{\Omega^{\#}}f^{\#}(x)\,u^{\#}(x)\,dx\\
& \leq \int_{\Omega^{\#}}f^{\#}(x)\,v(x)\,dx\\
&= \frac{\gamma(N,s)}{2}\int_{\R^{N}}\int_{\R^{N}}\frac{|v(x)-v(y)|^{2}}{|x-y|^{N+2s}}dx\,dy.
\end{align*}
\end{proof}

\begin{remark}\label{Limit s}\emph{
It is interesting to guess, in the spirit of \cite{BREZBOURMIR}, \cite{Bourgain200277}, \cite{mazsha},  what happens when we want somehow to pass to the limit as $s\rightarrow 1$ in the energy inequality \eqref{energineq}. We observe that by \cite[Theorem 1.2]{BICCARI} the solution $u$ to problem \eqref{eq.0} for $s=1$, \emph{i.e.} the solution to the local Poisson equation with homogeneous boundary condition $u=0$ on $\partial \Omega$, can be seen as the weak limit of the family of solutions $u_{s}$ to \eqref{eq.0} for $s\in (0,1)$. The same property holds for the solution $v$ to the symmetrized problem \eqref{eq.1} for $s=1$ and $v=0$ on $\partial \Omega^{\#}$. Then applying  Proposition  \ref{energies} to each couple of solutions $u_{s},\,v_{s}$ we can use \cite[Theorem 1.2]{BICCARI} to pass to the limit as $s\rightarrow1$ in the inequality
\[
[u_{s}]_{H^{s}(\R^{N})}\leq [v_{s}]_{H^{s}(\R^{N})},
\]
and obtain
\[
\int_{\Omega}|\nabla u|^{2}dx\leq \int_{\Omega^{\#}}|\nabla v|^{2}dx ,
\]
thus \eqref{energineq} can be seen as the nonlocal version of Talenti's energy inequality.}
\end{remark}
}
\section{Counterexamples to the pointwise comparison}\label{counter}
One could ask if the comparison in terms of mass concentration could be improved to give a pointwise estimate. In order to understand if a result similar to the one proved by Talenti \cite{Talenti1} in the local case could be expected also in the non local one, we devote this Section to discuss some special cases which give a negative answer, showing that a pointwise estimate cannot hold and then, that our result is optimal.

In the case $N=1$, $s=\frac12$, $\Omega=(-1,1)$, the explicit solution to problem
\begin{equation}\label{c.1}
\left\{
\begin{array}
[c]{lll}%
( -\Delta)^{1/2}u=f & & \text{in }%
\Omega,\\
\\
u=0 & & \text{on }\R\setminus\Omega,
\end{array}
\right. %
\end{equation}
can be computed explicitly in various cases, making use of the Green function on $\Omega$ (see, \emph{e.g.}, \cite[Theorem 3.1]{BUCUR})
\begin{equation}
u(x)=\frac1\pi\int_{-1}^1f(x)\log\left(\frac{1-x\,y+\sqrt{(1-x^2)(1-y^2)}}{|x-y|}\right)dy.\label{explicitform}
\end{equation}
A direct computation shows that the solutions $u_{(1)}$ and $u_{(2)}$ to \eqref{c.1} corresponding to the source terms $f_{(1)}(x)=|x|$ and $f_{(2)}(x)=\chi_{\frac12<|x|<1}(x)$, respectively, are given by
\begin{align*}
\notag u_{(1)}(x)=&\frac1\pi\left(\sqrt{1-x^2}-x^2\log|x|+x^2\log\left(1+\sqrt{1-x^2}\right)\right)
\displaybreak[1]\\
\notag\\
\notag u_{(2)}(x)=&\frac1\pi\left(\frac{2\pi}3\sqrt{1-x^2}-\left(\frac12-x\right)\log\left(\frac{1-\frac x2+\frac{\sqrt3}2\sqrt{1-x^2}}{|x-\frac12|}\right)-\left(\frac12+x\right)\log\left(\frac{1+\frac x2+\frac{\sqrt3}2\sqrt{1-x^2}}{|x+\frac12|}\right)\right).
\end{align*}

On the other hand, observing that $\Omega^\#=(-1,1)$, $f^\#_{(i)}(x)=1-f_{(i)}(x)$, $i=1,2$, and that $\sqrt{1-x^2}$ solves the above problem when the source term is 1 (see \cite{getoor} and \cite{Dyda}), the solution $v_{(i)}$, $i=1,2$, to the symmetrized problem
\begin{equation*}
\left\{
\begin{array}
[c]{lll}%
( -\Delta)^{1/2}v_{(i)}=f_{(i)}^\# & & \text{in }%
\Omega^\#,\\
\\
v_{(i)}=0 & & \text{on }\R\setminus\Omega^\#,
\end{array}
\right. %
\end{equation*}
is given by
\begin{equation*}
v_{(i)}(x)=\sqrt{1-x^2}-u_{(i)}(x)
\end{equation*}
All the functions $u_{(i)}$ and $v_{(i)}$ 
are symmetric with respect to the origin and one can prove that in both cases the pointwise estimate \eqref{comp} does not hold. These functions are plotted in Figure 1 and Figure 2 where they correspond to $u_{s,i}$ and $v_{s,i}$ with $s=\frac12$.

In the case $i=1$ (a similar  analysis can be carried out when $i=2$) a direct computation shows that
\begin{align*}
u_{(1)}(x)\sim&\frac{2}\pi\sqrt{1-x^2}\qquad \text{as }x\rightarrow 1
\displaybreak[1]\\
\notag\\
\notag
v_{(1)}(x)\sim&\frac{\pi-2}\pi\sqrt{1-x^2}\qquad \text{as }x\rightarrow 1,
\end{align*}
so, for a suitable $\varepsilon>0$ it holds
\begin{equation}\label{a1}
u_{(1)}(x)>v_{(1)}(x),\qquad1-\varepsilon\le|x|\le1
\end{equation}
On the other hand, being $f_{(1)}$ a nonnegative (not identically zero) function, we have (see, \emph{e.g.}, \cite[Theorem 2.3.3]{BUCVALD}) $u_{(1)}(x)>0$ in $[-1+\varepsilon,1-\varepsilon]$. Let
\begin{equation*}
\bar t=\min_{[-1+\varepsilon,1-\varepsilon]}u_{(1)}(x)>0.
\end{equation*}
We observe that $v_{(1)}(x)$ is an even function which strictly decreases for $x\in(0, 1)$. If $v_{(1)}(x)< \bar t$ then \eqref{a1} holds true for $x\in(-1,1)$ and then $u^\#(x)>v(x)$ in $(-1,1)$, contradicting \eqref{comp}.

So, we can suppose that there exists a unique $\bar x\in(0,1)$ such that $v_{(1)}(\bar x)=\bar t$.
We claim that the two truncated functions
\begin{equation*}
\bar u_{(1)}(x)=\min\{\bar t,u_{(1)}(x)\},\qquad \bar v_{(1)}(x)=\min\{\bar t,v_{(1)}(x)\},
\end{equation*}
are such that
\begin{equation}\label{ubar}
\bar u_{(1)}(x)\ge \bar v_{(1)}(x),\qquad x\in\R,
\end{equation}
with strict inequality in a left neighborough of $x=1$ (and in a right neighborough of $x=-1$).

Indeed, considering the cases where $\bar x$ belongs to $(0,1-\varepsilon)$ or to $[1-\varepsilon,1)$, from the definition of $\bar t$ it follows:
\begin{equation*}
u_{(1)}(\bar x)\ge\bar t.
\end{equation*}
Moreover, for $x\in(0,\bar x)$, in view of \eqref{a1} and the radial monotonicity of $v_{(1)}$ we have that $u_{(1)}(x)\geq \bar t $, thus $\bar u_{(1)}(x)=\bar t=\bar v_{(1)}(x)$, $x\in(0,\bar x)$. Then \eqref{a1} allows to obtain \eqref{ubar}.

By the definition of distribution function we have by \eqref{ubar}
\begin{equation*}
\mu_{u_{(1)}}(t)=\mu_{\bar u_{(1)}}(t)>\mu_{\bar v_{(1)}}(t)=\mu_{v_{(1)}}(t), \qquad t\in(0,\bar t),
\end{equation*}
so, in a left neighborough of $x = 1$ (and in a right neighborough of $x = -1$), it holds
$$u_{(1)}^\#(x)>v_{(1)}(x),$$
which proves that the pointwise estimate
$$u_{(1)}^\#(x)\le v_{(1)}(x),$$
does not hold true in $(-1,1)$.

\begin{figure}[t]\label{figura}
\begin{picture}(500,300)
    \put(-20,180){\includegraphics[width=8 cm]{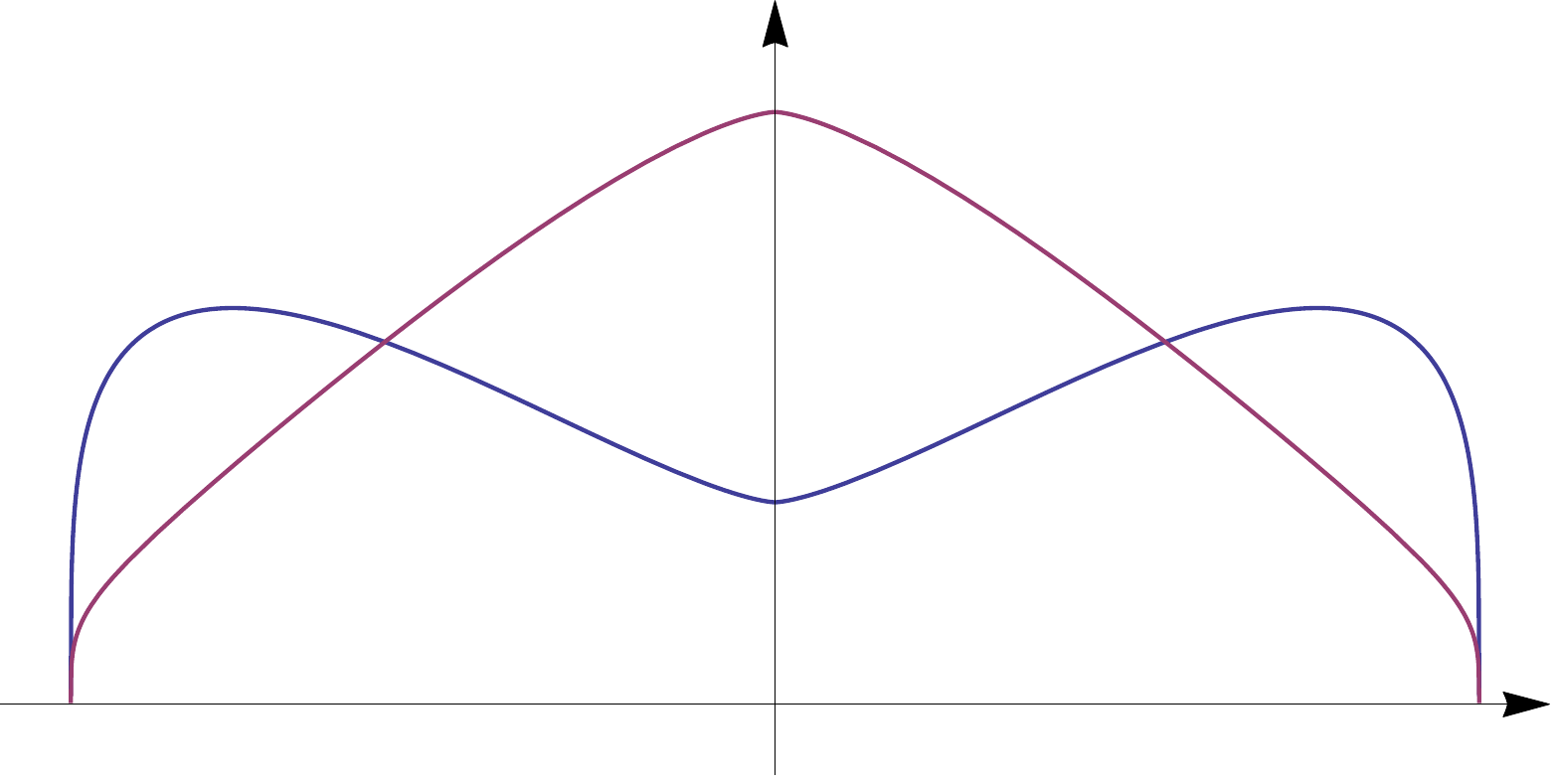}}
    \put(190,172){$1$}
        \put(90,170){$0$}
        \put(80,145){$s=1/4$}
        \put(-20,172){$-1$}
            \put(76,212){$u_{s,1}$}
    \put(76,282){$v_{s,1}$}
     \put(220,180){\includegraphics[width=8 cm]{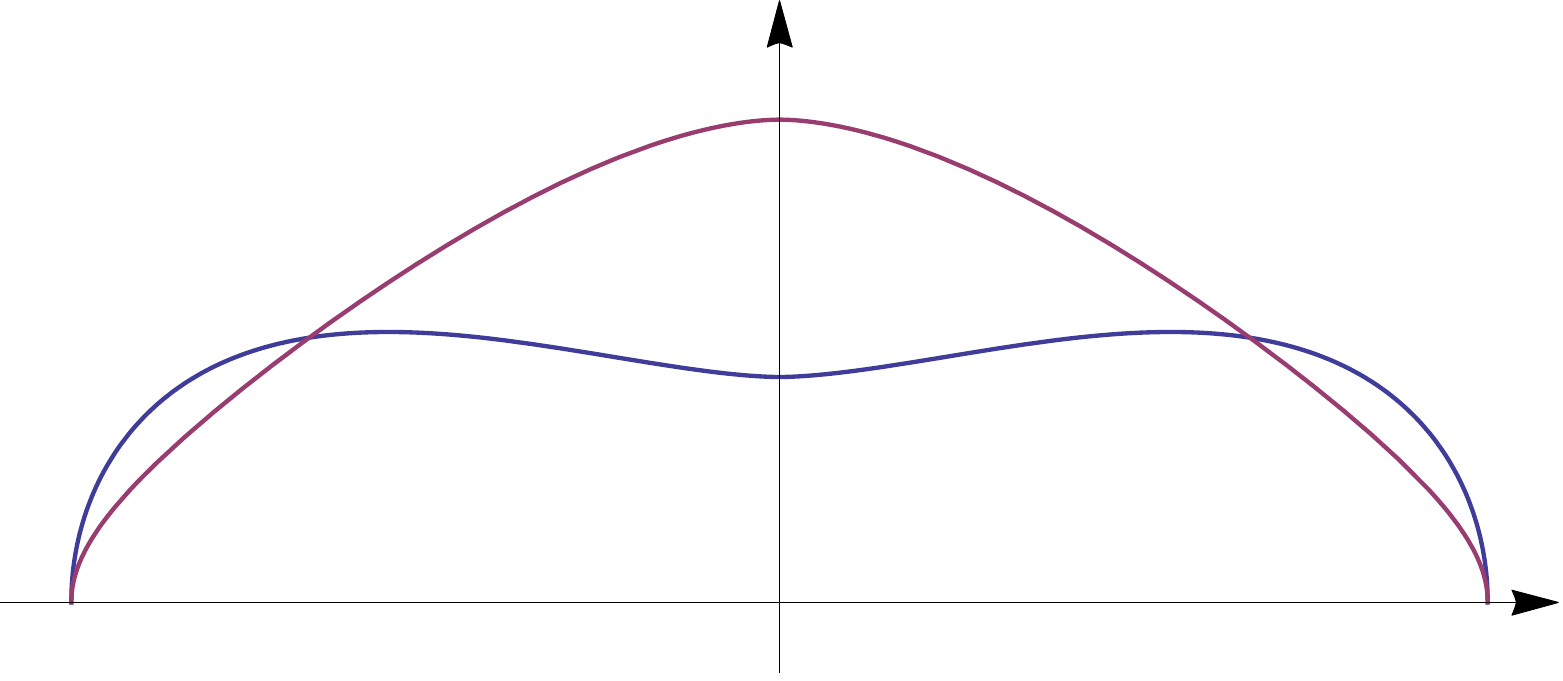}}
     \put(430,172){$1$}
        \put(328,170){$0$}
        \put(324,145){$s=1/2$}
        \put(220,172){$-1$}
            \put(316,214){$u_{s,1}$}
    \put(316,268){$v_{s,1}$}
 \put(-20,20){\includegraphics[width=8 cm]{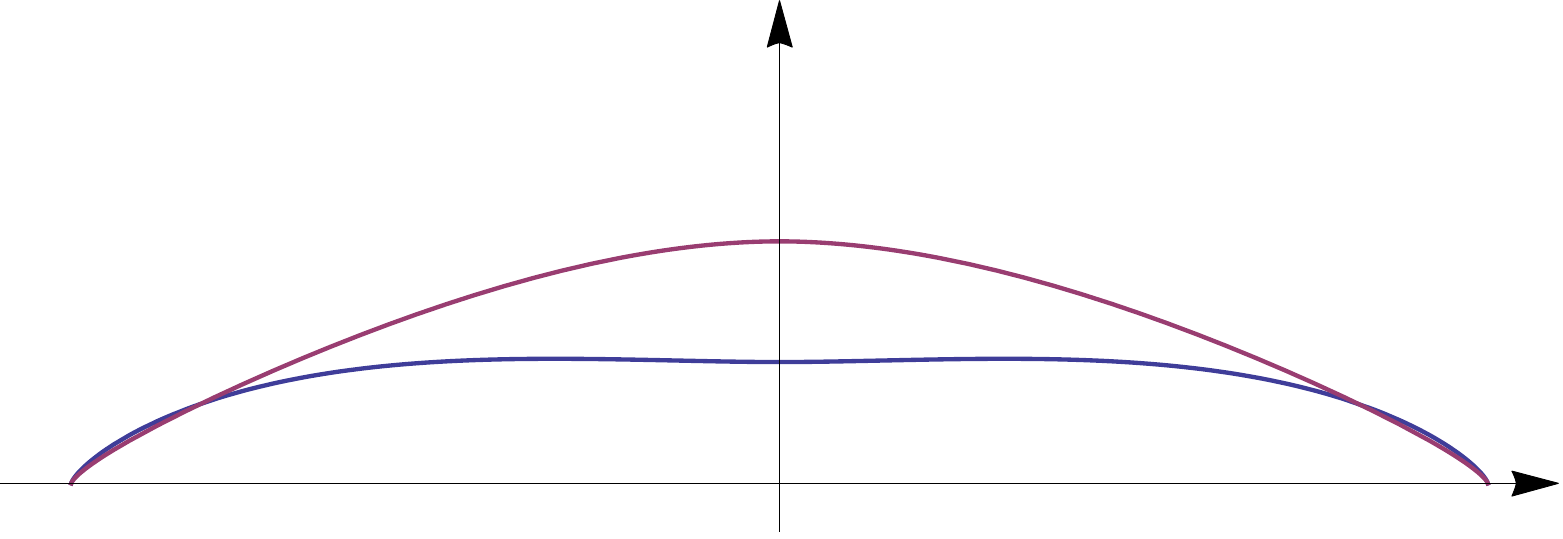}}
    \put(190,13){$1$}
        \put(90,12){$0$}
        \put(80,-1){$s=3/4$}
        \put(-20,13){$-1$}
            \put(76,35){$u_{s,1}$}
    \put(76,70){$v_{s,1}$}
\put(220,20){\includegraphics[width=8 cm]{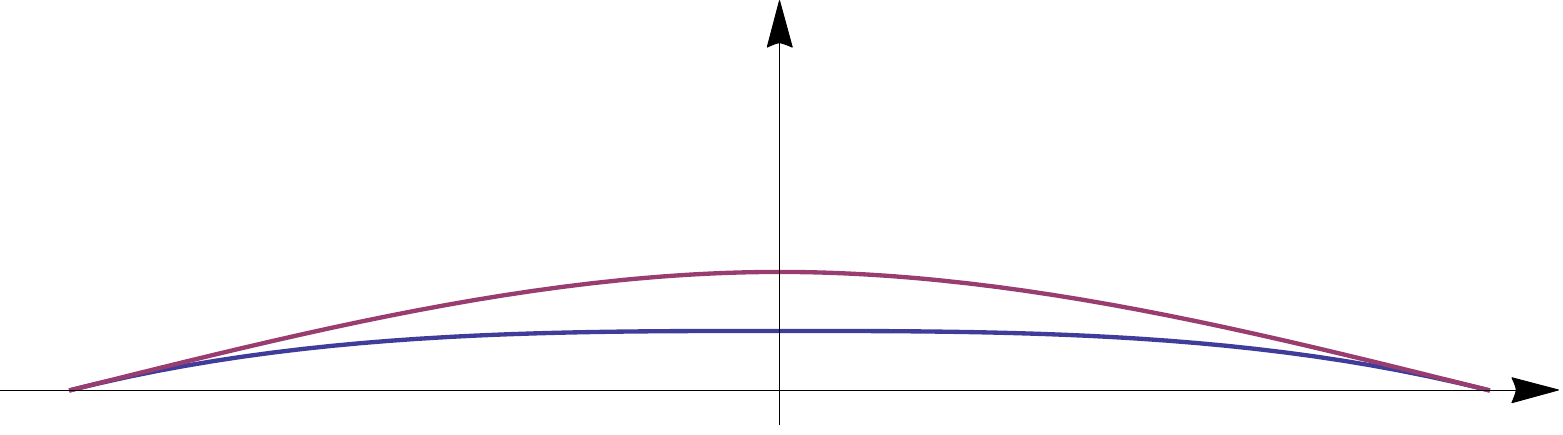}}
    \put(430,13){$1$}
        \put(330,12){$0$}
        \put(324,-1){$s=1$}
        \put(220,13){$-1$}
            \put(316,29){$u_{s,1}$}
    \put(316,50){$v_{s,1}$}
 \end{picture}%
 \caption{Plot of $u_{s,1}$ and $v_{s,1}$}
 \end{figure}
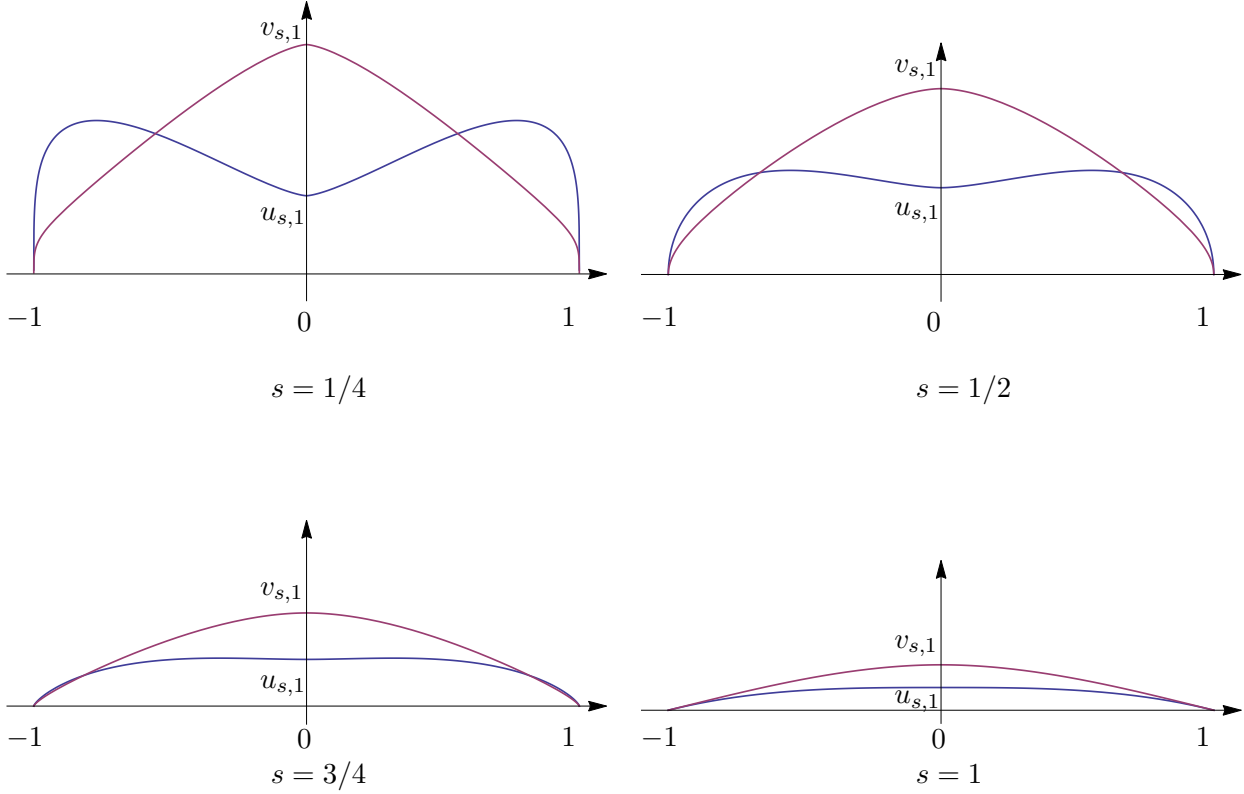

As a matter of fact the phenomenon described above happens to be true for every $s\in(0,1)$. Indeed, denote, respectively, by $u_{s,i}$ and $v_{s,i}$ the solutions to the problems
\begin{equation*}
\left\{
\begin{array}
[c]{lll}%
( -\Delta)^{s}u_{s,i}=f_{i} & & \text{in }%
\Omega,\\
\\
u_{s,i}=0 & & \text{on }\R\setminus\Omega,
\end{array}
\right. %
\qquad
\left\{
\begin{array}
[c]{lll}%
( -\Delta)^{s}v_{s,i}=f_{i}^\# & & \text{in }%
\Omega^\#,\\
\\
v_{s,i}=0 & & \text{on }\R\setminus\Omega^\#,
\end{array}
\right. %
\end{equation*}
where $\Omega$ and $f_{i}$, $i=1,2$ are like above. The following computations are made for $i=1$, similar arguments apply for $i=2$. For $s\in(0,1)$, $s\not=1/2$ the solution $u_{s,1}$ can be computed using the Green function on $\Omega$ (see again \cite[Theorem 3.1]{BUCUR}) obtaining
\begin{equation*}
u_{s,1}(x)=\frac1{2^{2s}\Gamma(s)^2}\int_{-1}^1|x|\,|z-x|^{2s-1}\left(\int_0^{\frac{(1-x^2)(1-z^2)}{|z-x|^2}}\frac{t^{s-1}}{(t+1)^{\frac12}}dt\right)dz
\end{equation*}
and it follows
\begin{align*}
\ell_{u_{s,1}}=&\lim_{x\rightarrow1^-}\frac{u_{s,1}(x)}{(1-x^2)^s}=\frac1{2^{2s}\Gamma(s)\Gamma(s+1)}\int_{-1}^1|z|
(1-z)^{s-1}(1+z)^sdz
\displaybreak[1]\\
\notag\\
\notag=&\frac1{2^{2s}\Gamma(s+1)^2}
\end{align*}
Recalling that  $(1-x^2)^s$ solves the  problem with the source term equal to $\Gamma(2s+1)$ we have
\begin{equation*}
v_{s,1}(x)=\frac{(1-x^2)^s}{\Gamma(2s+1)}-u_{s,1}(x)
\end{equation*}
and it follows
\begin{equation*}
\ell_{v_{s,1}}(s)=\lim_{x\rightarrow1^-}\frac{v_{s,1}(x)}{(1-x^2)^s}=\frac1{\Gamma(2s+1)}-\frac1{2^{2s}\Gamma(s+1)^2}
\end{equation*}
Now we use the following property of the function $\Gamma(x)$
\begin{equation*}
\sqrt\pi 2^{1-2x}\Gamma(2x)=\Gamma(x)\Gamma(x+\tfrac12),\qquad x>0,
\end{equation*}
and the definition of
beta function
\begin{equation*}
B(x,y)=\int_0^1t^{x-1}(1-t)^{y-1}dt=\frac{\Gamma(x)\Gamma(y)}{\Gamma(x+y)},
\end{equation*}
to get
\begin{align*}
\ell_{u_{s,1}}(s)-\ell_{v_{s,1}}(s)=&\frac2{2^{2s}\Gamma(s+1)^2}-\frac1{\Gamma(2s+1)}=\frac1{\Gamma(2s+1)}
\left(\frac2{\sqrt\pi}\frac{\Gamma(s+\frac12)}{\Gamma(s+1)}-1
\right)
\displaybreak[1]\\
\notag\\
\notag=&\frac1{\Gamma(2s+1)}
\left(\frac2{\pi}B(s+\tfrac12,\tfrac12)-1
\right)>0,\hskip3cm s\in(0,1),
\end{align*}
where we have observed that by definition $B(s+\tfrac12,\tfrac12)$ is decreasing with respect to $s\in(0,1)$ and $B(\tfrac32,\tfrac12)=\frac\pi2$.
Thus, as above, in a left neighborhood of $x=1$, it holds
$$u_{s,1}^\#(x)>v_{s,1}(x).$$
Now we can repeat the arguments used in the case $s=\frac12$ to show that the pointwise estimate
$$u_{s,1}^\#(x)\le v_{s,1}(x),$$
does not hold true in $(-1,1)$.

In Figure 1 and Figure 2 the functions $u_{s,i}$ and $v_{s,i}$ for various values of $s$ are plotted using the same unit for the two axes. The plots have been obtained applying the numerical tools offered by \textit{Mathematica} to the explicit 1D formula contained in \cite[Theorem 3.1]{BUCUR}
and proved to agree as expected with the outputs of the \textit{Matlab} code kindly provided by the authors of \cite{ABGOVAZ}.

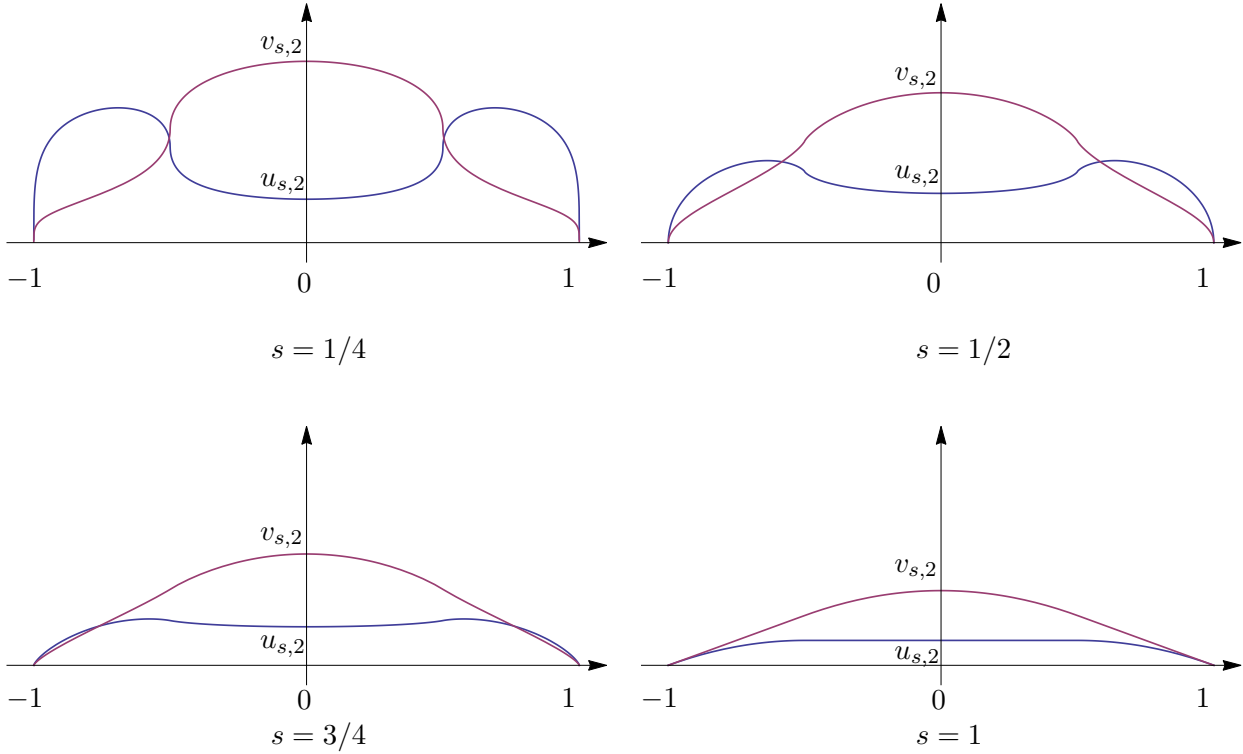
\begin{figure}[t]\label{Figura2}
\begin{picture}(200,200)
    \put(-20,180){\includegraphics[width=8 cm]{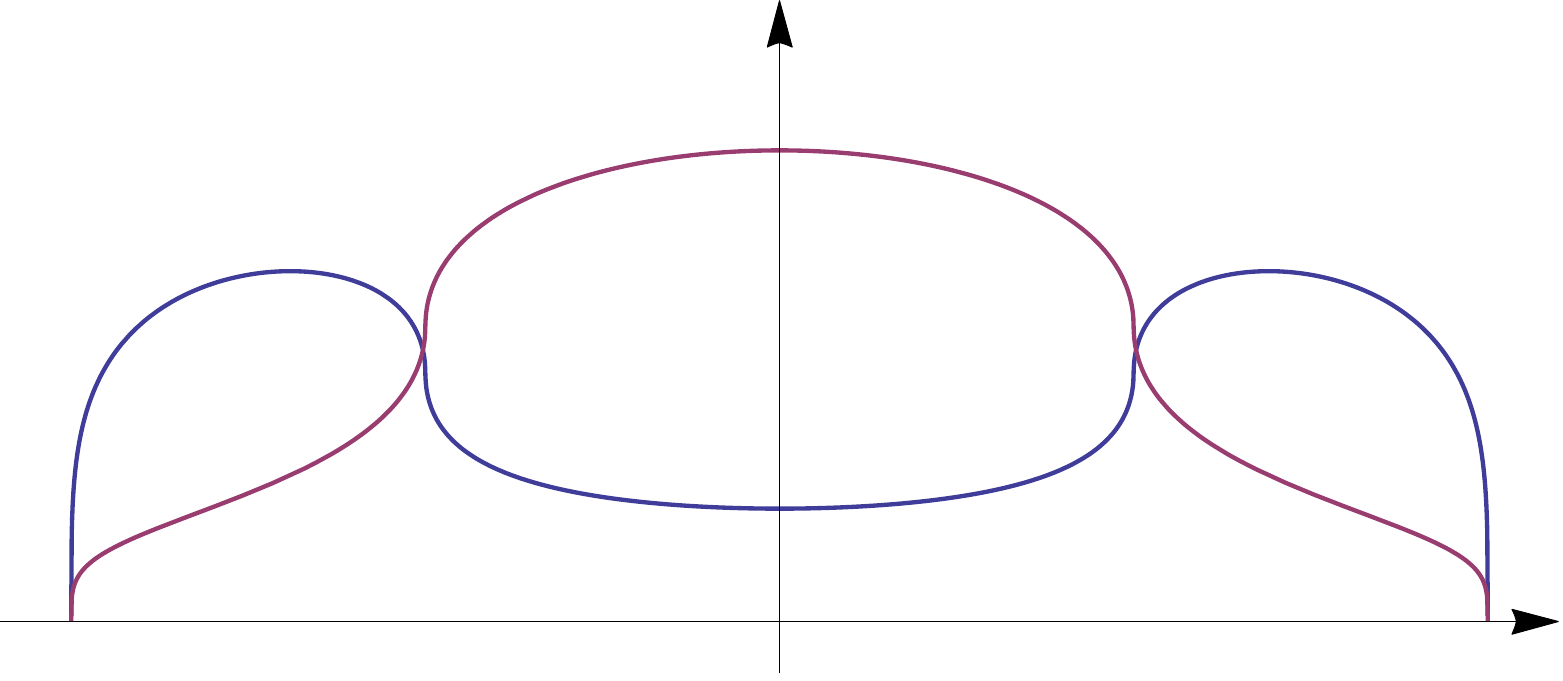}}
    \put(190,172){$1$}
        \put(90,170){$0$}
        \put(80,145){$s=1/4$}
        \put(-20,172){$-1$}
            \put(76,210){$u_{s,2}$}
    \put(76,262){$v_{s,2}$}
     \put(220,180){\includegraphics[width=8 cm]{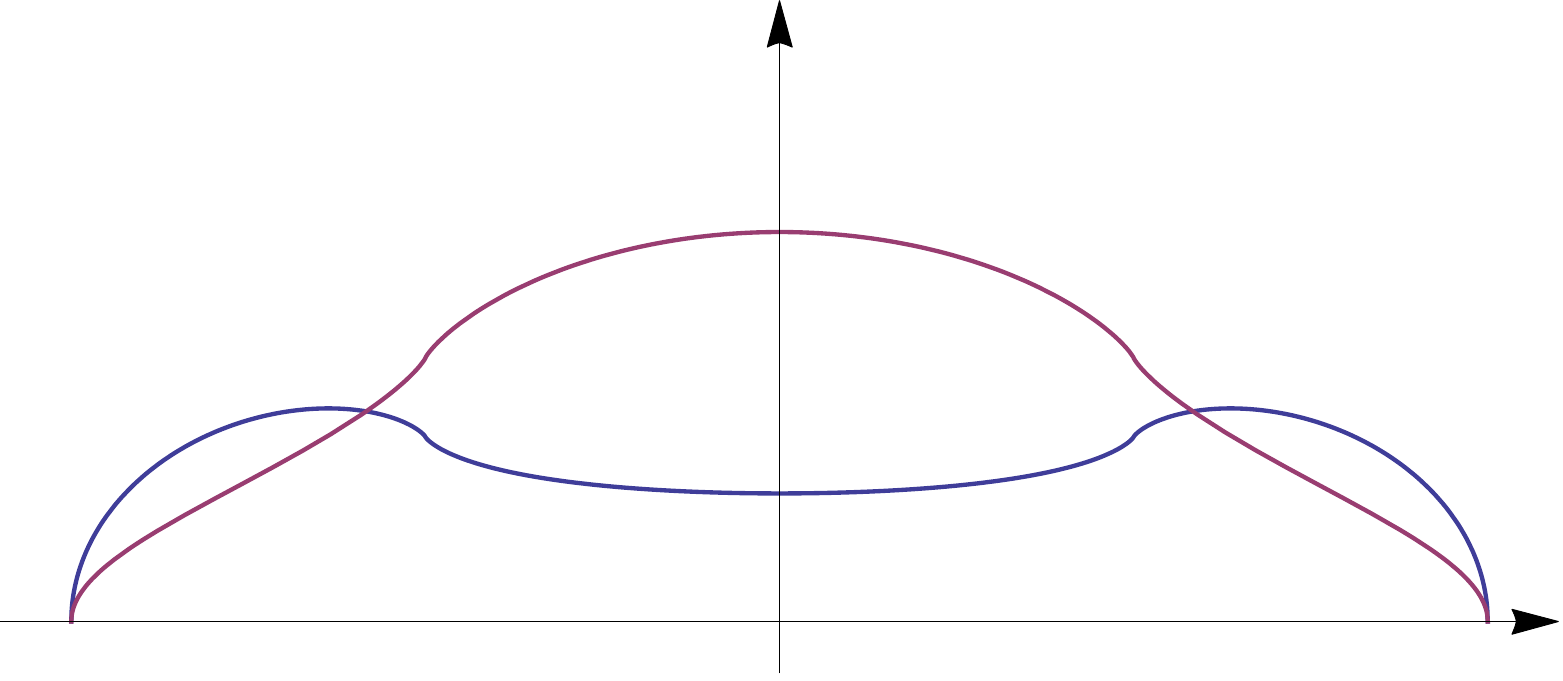}}
     \put(430,172){$1$}
        \put(328,170){$0$}
        \put(324,145){$s=1/2$}
        \put(220,172){$-1$}
            \put(316,213){$u_{s,2}$}
    \put(316,250){$v_{s,2}$}
 \put(-20,20){\includegraphics[width=8 cm]{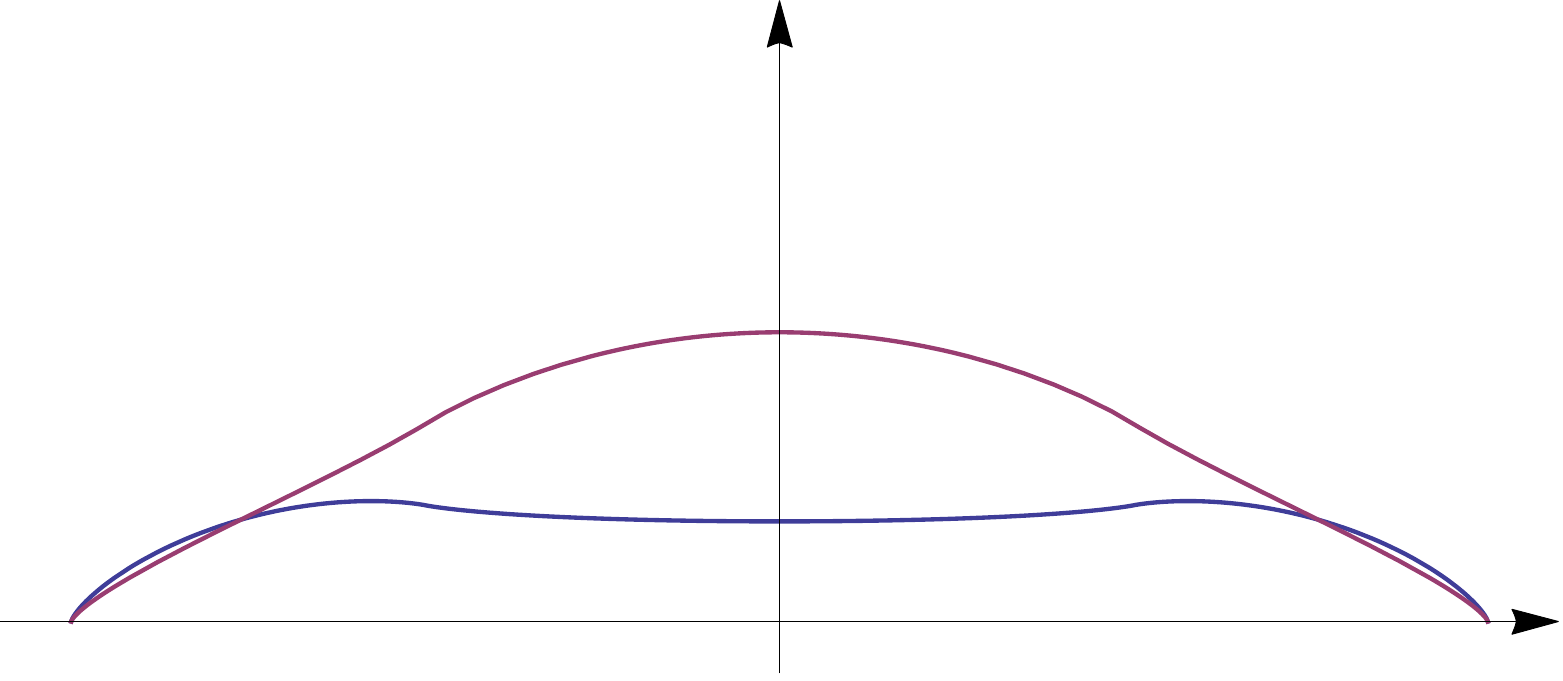}}
    \put(190,13){$1$}
        \put(90,12){$0$}
        \put(80,-1){$s=3/4$}
        \put(-20,13){$-1$}
            \put(76,35){$u_{s,2}$}
    \put(76,76){$v_{s,2}$}
\put(220,20){\includegraphics[width=8 cm]{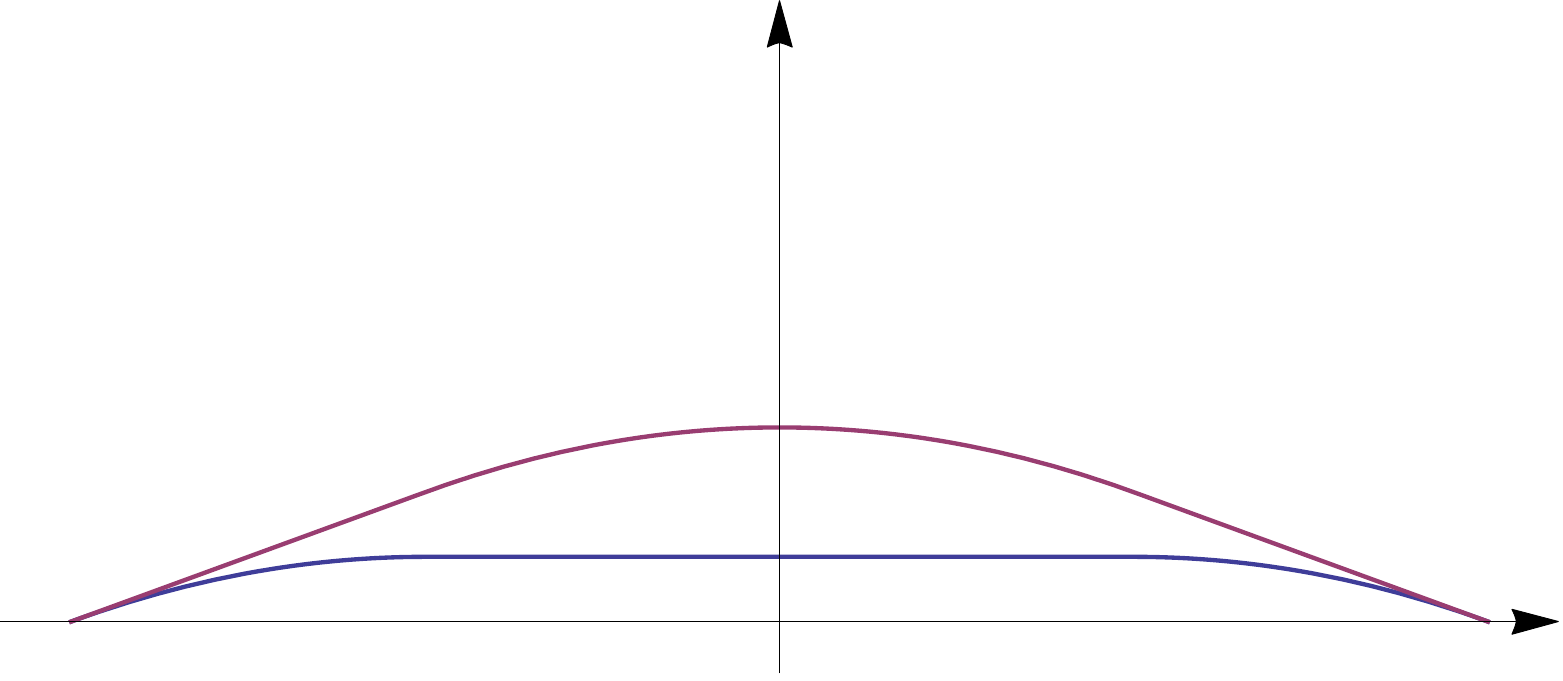}}
    \put(430,13){$1$}
        \put(330,12){$0$}
        \put(324,-1){$s=1$}
        \put(220,13){$-1$}
            \put(316,31){$u_{s,2}$}
    \put(316,64){$v_{s,2}$}
\end{picture}%
 \caption{Plot of $u_{s,2}$ and $v_{s,2}$}
 \end{figure}
\newpage
\section{Proof of Theorem \ref{main}}\label{Sec4}
\subsection{The case where $f$ is nonnegative and regular}

In the present subsection we suppose that $f\in C_{0}^{\infty}(\R^{N})$ and $f(x)\ge0$ in $\Omega$. It follows that $u$ is a continuous function on $\R^N$ of class $C^1$ in $\Omega$ and $0\le u(x)\le u_{\text{max}}<+\infty$ in $\Omega$, where $u_{\text{max}}=\max_{\R^N}u$ (see, \emph{e.g.}, \cite[Theorem 1.1]{GenRos}). We split the proof in four steps.
\vskip6pt
\noindent $\bullet$ {\textit{Step 1: Choice of the test function and nonlocal P\'olya-Sz{e}g\"o inequality}}.
\vskip6pt
\noindent For $0\le t<u_{\text{max}}$ and $h>0$, we choose the following test function\[
\varphi(x)=\mathcal{G}_{t,h}(u(x))
\]
where $\mathcal{G}_{t,h}(\theta)$ is the classical truncation
\begin{equation}
\mathcal{G}_{t,h}(\theta)  =\left\{
\begin{array}
[c]{lll}%
h &  & \text{if }\theta > t+h\\
&  & \\
\theta-t\, &  & \text{if }t< \theta \le t+h\\
&  & \\
0 &  & \text{if }\theta \leq t.\text{ }%
\end{array}
\right.\label{truncation}
\end{equation}
Notice that since
\[
|\mathcal{G}_{t,h}(u(x))-\mathcal{G}_{t,h}(u(y))|\leq\left |u(x)-u(y)\right|
\]
we immediately have that $\mathcal{G}_{t,h}(u)\in \mathcal{H}^{s}(\Omega)$, thus we can use $\mathcal{G}_{t,h}(u)$ as a test function in the weak formulation of problem \eqref{eq.0}.
Then we have
\begin{equation}
\frac{\gamma(N,s)}{2}\int_{\R^{N}}\int_{\R^{N}}\frac{\left(u(x)-u(y)\right)\left(\mathcal{G}_{t,h}(u(x))-\mathcal{G}_{t,h}(u(y))\right)}{|x-y|^{N+2s}}dx\,dy=
\int_{\Omega}f(x)\,\mathcal{G}_{t,h}(u(x))\,dx.\label{eqtest}
\end{equation}
In the spirit of \cite{lions} for the local case, our aim is now to find a bound from below of the left-hand side of \eqref{eqtest} in terms of the radially decreasing rearrangement $u^{\#}$.
More precisely, we will prove the inequality
\begin{align}\label{Polyatype}
&\int_{\R^{N}}\int_{\R^{N}}\frac{\left(u(x)-u(y)\right)\left(\mathcal{G}_{t,h}(u(x))-\mathcal{G}_{t,h}(u(y))\right)}{|x-y|^{N+2s}}dxdy
\\ &\geq
 \int_{\R^{N}}\int_{\R^{N}}\frac{\left(u^{\#}(x)-u^{\#}(y)\right)\left(\mathcal{G}_{t,h}(u^{\#}(x))-\mathcal{G}_{t,h}(u^{\#}(y))\right)}{|x-y|^{N+2s}}dxdy.\nonumber
\end{align}
{This approach consistently differs from the usual procedure in the local case, where Fleming-Rishel formula and isoperimetric inequality are employed and seem not to work in the present setting.}
Following \cite[Section 9]{ALIEB}, we write
\[
\int_{\R^{N}}\int_{\R^{N}}\frac{\left(u(x)-u(y)\right)\left(\mathcal{G}_{t,h}(u(x))-\mathcal{G}_{t,h}(u(y))\right)}{|x-y|^{N+2s}}dxdy
=\frac{1}{\Gamma(\frac{N+2s}{2})}\int_{0}^{\infty}I_{\alpha}[u,t,h]\,\alpha^{(N+2s)/2-1}d\alpha,
\]
where
\begin{equation}
I_{\alpha}[u,t,h]=\int_{\R^{N}}\int_{\R^{N}}\left(u(x)-u(y)\right)\left(\mathcal{G}_{t,h}(u(x))-\mathcal{G}_{t,h}(u(y))\right)\exp[-|x-y|^{2}\alpha]dx,dy.
\label{It}
\end{equation}
By virtue of this last representation, inequality \eqref{Polyatype} is proved when we succeed to show that
\begin{equation}
I_{\alpha}[u,t,h]\geq I_{\alpha}[u^{\#},t,h]\label{IneqI},
\end{equation}
for all $\alpha>0$. To this aim, we use Riesz's general rearrangement inequality \eqref{mainRieszineq} with the choice $W_{\alpha}(x)=\exp[-|x|^{2}\alpha]$, $a=1,\,b=-1$ and
\[
F(u,v)=u^{2}+v^{2}-(u-v)(\mathcal{G}_{t,h}(u)-\mathcal{G}_{t,h}(v))
\]
for all $u,\,v>0$. Observe that the function $F$ is eligible, since the fact that $\mathcal{G}_{t,h}(\theta)\leq \theta$ for all $\theta\geq0$, simple computations give
\[
F(u,v)=u(u-\mathcal{G}_{t,h}(u))+v(v-\mathcal{G}_{t,h}(v))+u\mathcal{G}_{t,h}(v)+v\mathcal{G}_{t,h}(u)\geq0
\]
and also \eqref{F} holds, because for given $u_{2}\geq u_{1}>0$ and $v_{2}\geq v_{1}>0$ one has
\[
F(u_{2},v_{2})+F(u_{1},v_{1})- F(u_{2},v_{1})-F(u_{1},v_{2})=(v_{1}-v_{2})(\mathcal{G}_{t,h}(u_1)-\mathcal{G}_{t,h}(u_2))
+(u_{1}-u_{2})(\mathcal{G}_{t,h}(v_{1})-\mathcal{G}_{t,h}(v_{2}))\geq0.
\]
Plugging such function $F$ in \eqref{mainRieszineq} yields
\[
\int_{\R^{N}}\int_{\R^{N}}F(u(x),u(y))\,W_{\alpha}(x-y)dx\,dy
\leq \int_{\R^{N}}\int_{\R^{N}}F(u^{\#}(x),u^{\#}(y))\,W_{\alpha}(x-y)dx\,dy
\]
that is
\begin{align}
&\int_{\R^{N}}\int_{\R^{N}}\left[u^{2}(x)+u^{2}(y)-(u(x)-u(y))\left(\mathcal{G}_{t,h}(u(x))\label{applRiesz}
-\mathcal{G}_{t,h}(u(y))\right)\right]W_{\alpha}(x-y)dx\,dy\\
& \leq
\int_{\R^{N}}\int_{\R^{N}}\left[(u^{\#})^{2}(x)+(u^{\#})^{2}(y)-(u^{\#}(x)-u^{\#}(y))\left(\mathcal{G}_{t,h}(u^{\#}(x))-
\mathcal{G}_{t,h}(u^{\#}(y))\right)\right]W_{\alpha}(x-y)dx\,dy.\nonumber
\end{align}
Notice that by the equimisurability property of rearrangements we have
\[
\int_{\R^{N}}\int_{\R^{N}} (u^{\#})^{2}(x)\,W_{\alpha}(x-y)dx\,dy=\int_{\R^{N}}\int_{\R^{N}} u^{2}(x)\,W_{\alpha}(x-y)dx\,dy
\]
then employing the symmetry of the kernel $W_{\alpha}$ in \eqref{applRiesz} we find \eqref{IneqI}.\\

So far we proved, as an easy consequence of \eqref{eqtest}, that
\begin{equation}
 \frac{\gamma(N,s)}{2}\int_{\R^{N}}\int_{\R^{N}}\frac{\left(u^{\#}(x)-u^{\#}(y)\right)\left(\mathcal{G}_{t,h}(u^{\#}(x))-\mathcal{G}_{t,h}(u^{\#}(y))\right)}{|x-y|^{N+2s}}dxdy
 \leq \int_{\Omega}f(x)\,\mathcal{G}_{t,h}(u(x))\,dx.\label{mainineq}
 \end{equation}

\vskip6pt
\noindent  $\bullet$ {\textit{Step 2: Rewriting \eqref{mainineq} in the radial coordinate}}.
\vskip6pt
\noindent This long step will be devoted in rewriting \eqref{mainineq} in terms of one-variable integrals with respect the radial coordinate $r:=|x|$. It will strongly involve the properties of the hypergeometric functions recalled in Section \ref{Hyperg}. First we set
 \[
 \mathfrc{u}(x)=\mathfrc{u}(|x|):=u^{\#}(x),
 \]
hence $\mathfrc{u}$ is a nonincreasing continuous function defined on $\R^N$ which vanishes for $|x|\ge R>0$, $R$ being the radius of $\Omega^{\#}$, and $t,h>0$. Furthermore, taking into account the fact
that Schwarz rearrangement preserves Lipschitz continuity (see, \emph{e.g.}, \cite[Theorem 2.3.3]{Kesavan}) because of the regularity of $u$, the function $\mathfrc{u}$ is \emph{locally} Lipschitz continuous in $\Omega^{\#}$, namely it is Lipschitz continuous in every $B_r(0)$ with $0<r<R$. For every $0\le t\le u_{\text{max}}$ there exists a unique $r(t)$ such that $|\{x:\mathfrc{u}(x)>t\}|=|B_{r(t)}(0)|$.
{
We observe that $r(t)$ is a non increasing right continuous function on $0\le t\le u_{\text{max}}$.} Our main issue now is \emph{how} to pass to the limit as $h\rightarrow 0$ in \eqref{mainineq}. To this aim, let us consider the following integral:
\begin{align}\allowdisplaybreaks
I_{t,h}= \frac1{N\omega_{N}}
\int_{\R^{N}}\int_{\R^{N}}\frac{\left(\mathfrc{u}(x)-\mathfrc{u}(y)\right)\left(\mathcal{G}_{t,h}(\mathfrc{u}(x))-\mathcal{G}_{t,h}(\mathfrc{u}(y))\right)}{|x-y|^{N+2s}}dxdy,
\label{integralh}
\end{align}
thus \eqref{mainineq} translates to
\[
 N\omega_{N}\frac{\gamma(N,s)}{2} I_{t,h}\leq \int_{\Omega}f(x)\,\mathcal{G}_{t,h}(u(x))\,dx.
\]

Putting as usual $r=|x|$, we have:
\begin{equation*}
I_{t,h}=\int_0^{+\infty}\left(\int_0^{+\infty}\big(\mathfrc{u}(r)-\mathfrc{u}(\rho)\big)\big(\mathcal{G}_{t,h}(\mathfrc{u}(r))-\mathcal{G}_{t,h}(\mathfrc{u}(\rho))\big)\Theta_{N,s}(r,\rho)\rho^{N-1}d\rho\right)r^{N-1}dr,
\end{equation*}
where
\begin{equation}\label{phi}
\Theta_{N,s}(r,\rho)=\frac1{N\omega_{N}}\int_{|x'|=1}\left(\int_{|y'|=1}\frac1{|r\,x'-\rho\,y'|^{N+2s}}dH^{N-1}(y')\right)dH^{N-1}(x')
\end{equation}
In order to compute $\Theta_{N,s}(r,\rho)$ we observe that the internal integral in \eqref{phi} does not depend on $x'$, so one can compute it choosing any fixed $x'$, obtaining
\begin{align}\allowdisplaybreaks\label{theta}
\Theta_{N,s}(r,\rho)=&\int_{|y'|=1}\frac1{|r\,x'-\rho\,y'|^{N+2s}}dH^{N-1}(y')\\
\notag\\
=&\alpha_N\int_{0}^{\pi}\frac{\sin^{N-2}\theta}{(r^{2}-2r\rho\cos\theta+\rho^{2})^{\frac{N+2s}2}}d\theta,\notag
\end{align}
where
$$\alpha_{N}=\frac{2\pi^{\frac{N-1}2}}{\Gamma(\frac{N-1}2)}.
$$

By definition $\Theta_{N,s}(r,\rho)$ is symmetric, that is,
$$\Theta_{N,s}(r,\rho)=\Theta_{N,s}(\rho,r),
$$
and, by \eqref{hyper}, it follows
\begin{equation}\label{explicit}
\Theta_{N,s}(r,\rho)=
\left\{
\begin{array}{ll}
\dfrac{\alpha_{N}}{\rho^{N+2s}}\>{}_{2}F_1\left(\dfrac{N+2s}2,s+1;\dfrac N2;\dfrac{r^{2}}{\rho^{2}}\right)  &\quad   \text{if }0\le r<\rho<+\infty \\
&\\
\dfrac{\alpha_{N}}{r^{N+2s}}\>{}_{2}F_1\left(\dfrac{N+2s}2,s+1;\dfrac N2;\dfrac{\rho^{2}}{r^{2}}\right) &   \quad   \text{if }0\le \rho<r<+\infty\\
\end{array}
\right.
\end{equation}
It is possible to obtain the asymptotic behaviour of $\Theta_{N,s}(r,\rho)$ as $r,\rho \rightarrow+\infty$ or $|r-\rho|\rightarrow0$. Indeed, using \eqref{zero} with $a=\frac{N+2s}2$, $b=s+1$, $c=\frac N2$, we have
\begin{equation}\label{infty}
\left\{
\begin{array}{ll}
\Theta_{N,s}(r,\rho)\sim\dfrac1{r^{N+2s}}&\qquad\text{ as }r\rightarrow+\infty\\
\\
\Theta_{N,s}(r,\rho)\sim\dfrac1{\rho^{N+2s}}&\qquad\text{ as }\rho\rightarrow+\infty.
\end{array}
\right.
\end{equation}
On the other hand, with the same choice of $a,\ b,\ c$,
\eqref{linear}
gives
\begin{equation*}
{}_{2}F_1\left(\dfrac{N+2s}2,s+1;\dfrac N2;x\right)= (1 - x)^{-1-2s}{}_{2}F_1\left(-s,\dfrac N2-s-1;\dfrac N2;x\right),
\end{equation*}
so by
\eqref{gauss} we get
\begin{equation}\label{uno}
\Theta_{N,s}(r,\rho)\sim\dfrac1{|r-\rho|^{1+2s}}\qquad\text{ as }|r-\rho|\rightarrow0.
\end{equation}

We also observe that \eqref{formulona}
tells us that the function $\Theta_{N,s}(r,\rho)$ and its derivatives can be related to the function $\Theta_{N+2,s}(r,\rho)$, Actually the use of that equality or a direct computation gives:
\begin{align}\allowdisplaybreaks\label{recura}
&\rho^{N-1}\Theta_{N,s}(r,\rho)=\frac{N+2s}{2\pi}\rho^{N+1}\Theta_{N+2,s}(r,\rho)+\frac 1{2\pi}
\frac{\partial}{\partial \rho}\left(\rho^{N}\Theta_{N+2,s-1}(r,\rho)\right)
\\
\notag\\
\label{recurb}
&\rho^{N-1}\frac{\partial\Theta_{N,s}}{\partial r}(r,\rho)=-\frac{N+2s}{2\pi}r
\frac{\partial}{\partial \rho}\left(\rho^N\Theta_{N+2,s}(r,\rho)\right)
\end{align}

Now we come back to the integral $I_{t,h}$ defined in \eqref{integralh} and observe that a decomposition which makes use of \eqref{truncation} gives:
\begin{align}\allowdisplaybreaks\label{verylongeq}
I_{t,h}=&\int_{r(t+h)}^{r(t)}\left(\int_{r(t+h)}^{r(t)}\big(\mathfrc{u}(r)-\mathfrc{u}(\rho)\big)^2\Theta_{N,s}(r,\rho)\rho^{N-1}d\rho\right)r^{N-1}dr\displaybreak[1]\\ \nonumber\\
&+2\int_0^{r(t+h)}\left(\int_{r(t+h)}^{r(t)}\big(\mathfrc{u}(r)-\mathfrc{u}(\rho)\big)\big(h-\mathfrc{u}(\rho)+t\big)\Theta_{N,s}(r,\rho)\rho^{N-1}d\rho\right)r^{N-1}dr\displaybreak[1]\nonumber\\ \nonumber\\
&+2\int_{r(t)}^{+\infty}\left(\int_{r(t+h)}^{r(t)}\big(\mathfrc{u}(r)-\mathfrc{u}(\rho)\big)\big(-\mathfrc{u}(\rho)+t\big)\Theta_{N,s}(r,\rho)\rho^{N-1}d\rho\right)r^{N-1}dr\displaybreak[1]\nonumber\\ \nonumber\\
&+2h\int_0^{r(t+h)}\left(\int_{r(t)}^{+\infty}\big(\mathfrc{u}(r)-\mathfrc{u}(\rho)\big)\Theta_{N,s}(r,\rho)\rho^{N-1}d\rho\right)r^{N-1}dr\displaybreak[1]\nonumber\\ \nonumber\\
&=I^{(1)}_{t,h}+2I^{(2)}_{t,h}+2I^{(3)}_{t,h}+2hI^{(4)}_{t,h}.\nonumber
\end{align}
Using the fact that $\mathfrc{u}(r)$ is locally Lipschitz continuous in $\Omega^{\#}$ and the asymptotic behaviour of $\Theta_{N,s}(r,\rho)$ as $r\rightarrow \rho$ and $r,\rho\rightarrow+\infty$, we can show that
\begin{equation}\label{inftya}
\frac1h I^{(i)}_{t,h}\xrightarrow{h\rightarrow0}0,\qquad i=1,2,3.
\end{equation}
Indeed, since $\mathfrc{u}$ is radially decreasing we find for $r,\rho\in (r(t+h),r(t))$, $$|(\mathfrc{u}(r)-\mathfrc{u}(\rho)|\leq [\mathfrc{u}(r(t+h))-\mathfrc{u}(r(t))]=h$$ then
\begin{align*}
&\frac1h I^{(1)}_{t,h}\le c\frac1h\int_{r(t+h)}^{r(t)}\left(\int_{r(t+h)}^{r(t)}|(\mathfrc{u}(r)-\mathfrc{u}(\rho)||r-\rho|^{-2s}d\rho\right)dr\\
&\leq c \int_{r(t+h)}^{r(t)}\left(\int_{r(t+h)}^{r(t)}|r-\rho|^{-2s}d\rho\right)dr
\xrightarrow{h\rightarrow0}0.
\end{align*}
Similarly we find
\begin{align*}
0&\le \frac1h I^{(2)}_{t,h}\le c\int_0^{r(t+h)}\left(\int_{r(t+h)}^{r(t)}|r-\rho|^{-2s}d\rho\right)dr\\
& \le c |r(t)-r(t+h)|^{2-2s}
\xrightarrow{h\rightarrow0}0,
\end{align*}
while
\begin{align*}
 \frac1h I^{(3)}_{t,h}&\le \frac{c}{h} \int_{r(t)}^{R}\left(\int_{r(t+h)}^{r(t)}|\mathfrc{u}(\rho)-t||r-\rho|^{-2s}d\rho\right)dr+\frac1h\int_{R}^{+\infty}\left(\int_{r(t+h)}^{r(t)}|\mathfrc{u}(\rho)-t|\Theta_{N,s}(r,\rho)\rho^{N-1}d\rho\right)r^{N-1}dr\\
 &\leq c \left[\int_{r(t)}^{R}\left(\int_{r(t+h)}^{r(t)}|\mathfrc{u}(\rho)-t||r-\rho|^{-2s}d\rho\right)dr
 + \int_{r(t+h)}^{r(t)} \rho^{N-1} \left(\int_{R}^{+\infty}\Theta_{N,s}(r,\rho)r^{N-1}dr\right)d\rho\right]
 \xrightarrow{h\rightarrow0}0.
\end{align*}

Then, collecting \eqref{verylongeq} and \eqref{inftya}, for the left hand side of \eqref{mainineq} it holds, for every $t\in(0,u_\text{max})$,
\begin{align*}\allowdisplaybreaks
&\lim_{h\rightarrow 0^+}\frac{\gamma(N,s)}{2h}\int_{\R^{N}}\int_{\R^{N}}\frac{\left(\mathfrc{u}(x)-\mathfrc{u}(y)\right)\left(\mathcal{G}_{t,h}(\mathfrc{u}(x))-\mathcal{G}_{t,h}(\mathfrc{u}(y))\right)}{|x-y|^{N+2s}}
=\displaybreak[1]\\
&\quad=N\omega_{N}\gamma(N,s)\lim_{h\rightarrow 0^+} I_{t,h}^{(4)}
\\
&\quad=N\omega_{N}\gamma(N,s)\lim_{h\rightarrow 0^+}
\int_0^{r(t+h)}\left(\int_{r(t)}^{+\infty}\big(\mathfrc{u}(r)-\mathfrc{u}(\rho)\big)\Theta_{N,s}(r,\rho)\rho^{N-1}d\rho\right)r^{N-1}dr
\displaybreak[1]\\
\\
&\quad=N\omega_{N}\gamma(N,s)
\int_0^{r(t)}\left(\int_{r(t)}^{+\infty}\big(\mathfrc{u}(r)-\mathfrc{u}(\rho)\big)\Theta_{N,s}(r,\rho)\rho^{N-1}d\rho\right)r^{N-1}dr
\end{align*}
where the last passage to the limit is justified by monotone convergence and by the fact that the function
\begin{equation*}
(r,\rho)\rightarrow \big(\mathfrc{u}(r)-\mathfrc{u}(\rho)\big)\Theta_{N,s}(r,\rho)\rho^{N-1}r^{N-1}
\end{equation*}
is summable on $\big(0,r(t)\big)\times\big(r(t),+\infty\big)$. Indeed, for big values of $\rho$ one uses the asymptotic behaviour of $\Theta_{N,s}(r,\rho)$, while for $r-\rho\rightarrow0$ one uses the fact that
\begin{equation*}
\big|\big(\mathfrc{u}(r)-\mathfrc{u}(\rho)\big)\Theta_{N,s}(r,\rho)\big|\le c|r-\rho|^{-2s}.
\end{equation*}
On the other hand,
\begin{equation*}
\int_{\Omega}f(x)\,\mathcal{G}_{t,h}(u(x))\,dx\xrightarrow{h\rightarrow0}\int_{u>t}f(x)\,dx\le
N\omega_N\int_0^{r(t)}f^{*}(\omega_{N}\rho^{N})\rho^{N-1}d\rho.
\end{equation*}
So, for  every $t\in(0,u_\text{max})$, \eqref{mainineq} implies
\begin{equation}\label{ineq}
\gamma(N,s)
\int_0^{r(t)}\left(\int_{r(t)}^{+\infty}\big(\mathfrc{u}(r)-\mathfrc{u}(\rho)\big)\Theta_{N,s}(r,\rho)\rho^{N-1}d\rho\right)r^{N-1}dr
\le \int_0^{r(t)}f^{*}(\omega_{N}\rho^{N})\rho^{N-1}d\rho
\end{equation}
We now change the variables in both integrals, putting $r^N=\mathfrc{r}$ and $\rho^N=\mathfrc{s}$,
to obtain
\begin{equation}\label{inequa}
\frac{\gamma(N,s)}{N}
\int_0^{r(t)^N}\left(\int_{r(t)^N}^{+\infty}\big(\mathfrc{u}(\mathfrc{r}^{\frac1N})-\mathfrc{u}(\mathfrc{s}^{\frac1N})\big)\Theta_{N,s}(\mathfrc{r}^{\frac1N},\mathfrc{s}^{\frac1N})d\mathfrc{s}\right)d\mathfrc{r}
\le \int_0^{r(t)^N}f^{*}(\omega_{N}\mathfrc{s}^{\frac1N})d\mathfrc{s}
\end{equation}

For $\sigma\in (0,+\infty)$ we consider the functions
\begin{equation*}
H(\sigma)=\int_0^{\sigma}\left(\int_{\sigma}^{+\infty}\big(\mathfrc{u}(\mathfrc{r}^{\frac1N})-\mathfrc{u}(\mathfrc{s}^{\frac1N})\big)\Theta_{N,s}(\mathfrc{r}^{\frac1N},\mathfrc{s}^{\frac1N})\,d\mathfrc{s}\right)d\mathfrc{r}
\end{equation*}
and
\begin{equation*}
G(\sigma)=\int_0^{\sigma}f^{*}(\omega_{N}\mathfrc{s}^{\frac1N})\,d\mathfrc{s}.
\end{equation*}
By the asymptotic behaviors of $\Theta_{N,s}$, both functions $H(\sigma)$ and $G(\sigma)$ are continuous and inequality \eqref{inequa} can be written as
\begin{equation*}
\frac{\gamma(N,s)}{N}
H(\sigma)
\le  G(\sigma),\qquad \text{if }\sigma=r(t)^N \text{ for some }t\ge0
\end{equation*}
Our aim is to show that the above inequality holds true for every $\sigma\in (0,+\infty)$. Two cases are missing:

\vskip6pt
\noindent  (a) $\sigma>r(0)^N$;
\vskip5pt
\noindent  (b) $r(t)^N<\sigma\le r(t^-)^N$, if $u$ has a flat zone at the level $t>0$.
\vskip6pt
%
In case (a) we have (recall that $u=0$ in $\R^{N}\setminus \Omega$):
\begin{equation*}
H(\sigma)=\int_0^{r(0)^N}\left(\int_{\sigma}^{+\infty}\mathfrc{u}(\mathfrc{r}^{\frac1N})\Theta_{N,s}(\mathfrc{r}^{\frac1N},\mathfrc{s}^{\frac1N})\,d\mathfrc{s}\right)d\mathfrc{r}
\end{equation*}
which implies that $H(\sigma)$ is a non increasing function for $\sigma>r(0)^N$. On the other hand, $G(\sigma)$ is constant in the same interval, so, taking into account the fact that $\frac{\gamma(N,s)}{N}
H(r(0)^N)
\le  G(r(0)^N)$, we have
\begin{equation*}
\frac{\gamma(N,s)}{N}
H(\sigma)
\le  G(\sigma),\qquad \sigma> r(0)^N.
\end{equation*}

If $u$ does not have a flat zone we have finished, so let us consider case (b), that is, let $t$ be such that $|\{x:u(x)=t\}|>0$. The claimed inequality holds true at $\sigma=r(t)^N$ and, by continuity, also at $\sigma=r(t^-)^N$,
that is
\begin{equation*}
\frac{\gamma(N,s)}{N}
H(r(t^-)^N)
\le  G(r(t^-)^N).
\end{equation*}
It is immediate to observe that the function $G(\sigma)$ is concave, so, in order to prove the claimed inequality for $r(t)^N<\sigma< r(t^-)^N$, it is sufficient to show that $H(\sigma)$ is convex on such an interval. Indeed, 
it holds
\begin{align*}\allowdisplaybreaks
H'(\sigma)&=\int_{\sigma}^{+\infty}\big(\mathfrc{u}(\sigma^{\frac1N})-\mathfrc{u}(\mathfrc{s}^{\frac1N})\big)\Theta_{N,s}(\sigma^{\frac1N},\mathfrc{s}^{\frac1N})\,d\mathfrc{s}
-\int_{0}^{\sigma}\big(\mathfrc{u}(\mathfrc{r}^{\frac1N})-\mathfrc{u}(\sigma^{\frac1N})\big)\Theta_{N,s}(\mathfrc{r}^{\frac1N},\sigma^{\frac1N})\,d\mathfrc{r}\displaybreak[1]\\
\\
&=\int_{r(t^-)^N}^{+\infty}\big(t-\mathfrc{u}(\mathfrc{s}^{\frac1N})\big)\Theta_{N,s}(\sigma^{\frac1N},\mathfrc{s}^{\frac1N})d\mathfrc{s}
-\int_{0}^{r(t)^N}\big(\mathfrc{u}(\mathfrc{r}^{\frac1N})-t\big)\Theta_{N,s}(\mathfrc{r}^{\frac1N},\sigma^{\frac1N})d\mathfrc{r}\displaybreak[2]\\
\\
&=:\mathsf{H_{1}}(\sigma)-\mathsf{H_{2}}(\sigma).
\end{align*}
We notice that for $r(t)^N<\sigma< r(t^-)^N$, using \eqref{explicit} a direct computation shows that
\[
\mathsf{H_{1}}^{\prime}(\sigma)=\int_{r(t^-)^N}^{+\infty}\big(t-\mathfrc{u}(\mathfrc{s}^{\frac1N})\big)\frac{\partial}{\partial \sigma}\left(\Theta_{N,s}(\sigma^{\frac1N},\mathfrc{s}^{\frac1N})\right)d\mathfrc{s}\geq0,
\]
\[
\mathsf{H_{2}}^{\prime}(\sigma)=\int_{0}^{r(t)^N}\big(\mathfrc{u}(\mathfrc{r}^{\frac1N})-t\big)\frac{\partial}{\partial \sigma}\left(\Theta_{N,s}(\sigma^{\frac1N},\mathfrc{r}^{\frac1N})\right)d\mathfrc{r}\leq0,
\]
then $H'(\sigma)$ is increasing in $r(t)^N<\sigma< r(t^-)^N$, implying that $H(\sigma)$ is convex.

Thus, we have proved
\begin{equation*}
\frac{\gamma(N,s)}{N}
H(\sigma)
\le  G(\sigma),\qquad \sigma\ge0,
\end{equation*}
and, performing again a change of variables, for every $r\ge 0$, we get
\begin{equation}\label{inequality}
\gamma(N,s)
\int_0^{r}\left(\int_{r}^{+\infty}\big(\mathfrc{u}(\tau)-\mathfrc{u}(\rho)\big)\Theta_{N,s}(\tau,\rho)\rho^{N-1}d\rho\right)\tau^{N-1}d\tau
\le\int_0^{r}f^{*}(\omega_{N}\rho^{N})\rho^{N-1}d\rho.
\end{equation}
\vskip6pt
\noindent $\bullet$ {\textit{Step 3: Rewriting \eqref{inequality} in terms of the spherical mean function}}.
\vskip6pt
\noindent Now our goal is to \emph{rewrite} the left-hand side of \eqref{inequality} in terms of the following spherical mean function
\[
U(x)=U(|x|)=\frac1{|x|^N}\int_0^{|x|}\mathfrc{u}(\rho)\rho^{N-1}d\rho,
\]
defined for all $x\in \R^N$. It is a very important step and represents one of the main novelty of this approach, allowing to represent left-hand side of \eqref{inequality} as the fractional Laplacian of $U$ in $N+2$ variables. It is easy to show that:
\begin{equation*}
0\le \mathfrc{u}(x)\le u_\text{max},\qquad 0\le U(x)\le \frac{u_\text{max}}N
\end{equation*}
and
\begin{equation*}
\mathfrc{u}(x)\le N\,U(x),\qquad x\in\R^N
\end{equation*}

Furthermore
\begin{equation*}
U'(\rho)=\frac{\mathfrc{u}(\rho)}{\rho}-N\frac{U(\rho)}{\rho}\le0, \qquad\rho>0
\end{equation*}
So $U'(\rho)$ is of class $C^{1,1_{loc}}$ and $U(\rho)$ is strictly decreasing for $\rho\ge r_0$ where $r_0\ge0$ is such that
\begin{equation*}
\omega_Nr_0^N=|\{x\in \R^N:u(x)=u_\text{max}\}|
\end{equation*}
Indeed, by definition,
\[
\begin{array}{ll}
U(\rho)=\dfrac{u_\text{max}}N  &   0\le\rho\le r_0 \\
\\
U'(\rho)<0  &   \rho> r_0 \\
\end{array}
\]
Turning back to \eqref{inequality}, we write the left integral as an improper integral, \emph{i.e.} in the form
\[
\int_0^{r}\left(\int_{r}^{+\infty}\big(\mathfrc{u}(\tau)-\mathfrc{u}(\rho)\big)\Theta_{N,s}(\tau,\rho)\rho^{N-1}d\rho\right)\tau^{N-1}d\tau
=\lim_{\varepsilon\rightarrow 0^{+}}\mathsf{I}_{\varepsilon}
\]
where, for $\varepsilon>0$,
\[
\mathsf{I}_{t,\varepsilon}:=
\int_0^{r-\varepsilon}\left(\int_{r}^{+\infty}\big(\mathfrc{u}(\tau)-\mathfrc{u}(\rho)\big)\Theta_{N,s}(\tau,\rho)\rho^{N-1}d\rho\right)\tau^{N-1}d\tau
\]
and we manipulate $\mathsf{I}_{\varepsilon}$ in order to involve the function $U$. This approach has the advantage to avoid considerations regarding the strong singularity of $\Theta_{N,s}(\tau,\rho)$, when $\tau=\rho$, in the boundary terms appearing in the integration by parts formulas.

Splitting the integral $\mathsf{I}_{t,\varepsilon}$ we have

\begin{align*}
\mathsf{I}_{t,\varepsilon}&=\int_0^{r-\varepsilon}\mathfrc{u}(\tau)\left(\int_{r}^{+\infty}\Theta_{N,s}(\tau,\rho)\rho^{N-1}d\rho\right)\tau^{N-1}d\tau-
\int_0^{r-\varepsilon}\left(\int_{r}^{+\infty}\mathfrc{u}(\rho)\Theta_{N,s}(\tau,\rho)\rho^{N-1}d\rho\right)\tau^{N-1}d\tau\\
&=\int_{r}^{+\infty}\rho^{N-1}\left(\int_0^{r-\varepsilon}\mathfrc{u}(\tau)\Theta_{N,s}(\tau,\rho)\tau^{N-1}\right)d\rho-
\int_0^{r-\varepsilon}\left(\int_{r}^{+\infty}\mathfrc{u}(\rho)\Theta_{N,s}(\tau,\rho)\rho^{N-1}d\rho\right)\tau^{N-1}d\tau\\
\end{align*}
and since
\[
\mathfrc{u}(\tau)\tau^{N-1}=\frac{d}{d\tau}\int_{0}^{\tau}\mathfrc{u}(\sigma)\sigma^{N-1}d\sigma
\]
integrating by parts we get
\begin{align}\allowdisplaybreaks\label{quattro}
\mathsf{I}_{t,\varepsilon}=&(r-\varepsilon)^N U(r-\varepsilon)\int_{r}^{+\infty}\Theta_{N,s}(r-\varepsilon,\rho)\rho^{N-1}d\rho\displaybreak[1]\\
\notag\\
\notag&+r^NU(r)\int_0^{r-\varepsilon}\Theta_{N,s}(\tau,r)\tau^{N-1}d\tau\displaybreak[1]\\
\notag\\
\notag&
-\int_{r}^{+\infty}\left(\int_0^{r-\varepsilon}U(\tau)\frac{\partial\Theta_{N,s}}{\partial \tau}(\tau,\rho)\tau^Nd\tau\right)\rho^{N-1}d\rho
\displaybreak[1]\\
\notag\\
\notag&
+\int_0^{r-\varepsilon}\left(\int_{r}^{+\infty}U(\rho)\frac{\partial\Theta_{N,s}}{\partial \rho}(\tau,\rho)\rho^Nd\rho\right)\tau^{N-1}d\tau.
\end{align}
Now we evaluate each integral above. For the first two integrals we use  \eqref{recura} to get
\begin{align}\allowdisplaybreaks\label{inta}
&(r-\varepsilon)^NU(r-\varepsilon)\int_{r}^{+\infty}\Theta_{N,s}(r-\varepsilon,\rho)\rho^{N-1}d\rho\displaybreak[1]\\
\notag\\
\notag&=\frac{N+2s}{2\pi}(r-\varepsilon)^NU(r-\varepsilon)\int_{r}^{+\infty}\Theta_{N+2,s}(r-\varepsilon,\rho)\rho^{N+1}d\rho+\displaybreak[1]\\
\notag\\
\notag&\qquad-\frac{1}{2\pi}r^{N}(r-\varepsilon)^NU(r-\varepsilon)\Theta_{N+2,s-1}(r-\varepsilon,r),
\end{align}
and (recall that $\Theta_{N,s}$ is symmetric)\begin{align}\allowdisplaybreaks\label{intb}
&r^NU(r)\int_0^{r-\varepsilon}\Theta_{N,s}(\tau,r)\tau^{N-1}d\tau=\displaybreak[1]\\
\notag\\
\notag&\quad=\frac{N+2s}{2\pi}r^NU(r)\int_0^{r-\varepsilon}\Theta_{N+2,s}(\tau,r)\tau^{N+1}d\tau+\displaybreak[1]\\
\notag\\
\notag&\qquad+\frac{1}{2\pi}r^N(r-\varepsilon)^{N}U(r)\Theta_{N+2,s-1}(r-\varepsilon,r)
\end{align}

For the remaining two integrals in \eqref{quattro} we use  \eqref{recurb} to get
\begin{align}\allowdisplaybreaks\label{intc}
\int_{r}^{+\infty}\left(\int_0^{r-\varepsilon}U(\tau)\frac{\partial\Theta_{N,s}}{\partial \tau}(\tau,\rho)\tau^Nd\tau\right)\rho^{N-1}d\rho&=\int_0^{r-\varepsilon}\left(U(\tau)\int_{r}^{+\infty}\frac{\partial\Theta_{N,s}}{\partial \tau}(\tau,\rho)\rho^{N-1}d\rho\right)\tau^Nd\tau\displaybreak[1]\\
\notag\\
\notag=&\frac{N+2s}{2\pi}r^N\int_0^{r-\varepsilon}U(\tau)\Theta_{N+2,s}(\tau,r)\tau^{N+1}d\tau,
\end{align}

and
\begin{align}\allowdisplaybreaks\label{intd}
\int_0^{r-\varepsilon}\left(\int_{r}^{\infty}U(\rho)\frac{\partial\Theta_{N,s}}{\partial \rho}(\tau,\rho)\rho^Nd\rho\right)\tau^{N-1}d\tau&=\int_{r}^{+\infty}\left(U(\rho)\int_0^{r-\varepsilon}\frac{\partial\Theta_{N,s}}{\partial \rho}(\tau,\rho)\tau^{N-1}d\tau\right)\rho^Nd\rho\displaybreak[1]\\
\notag \\
\notag=&-\frac{N+2s}{2\pi}(r-\varepsilon)^N\int_{r}^{+\infty}U(\rho)\Theta_{N+2,s}(r-\varepsilon,\rho)\rho^{N+1}d\rho.
\end{align}

Collecting \eqref{quattro}-\eqref{intd} we have:
\begin{align}\allowdisplaybreaks\label{quattr}
\mathsf{I}_{t,\varepsilon}&=\frac{N+2s}{2\pi}\Biggl((r-\varepsilon)^N
\int_{r}^{+\infty}\big(U(r-\varepsilon)-U(\rho)\big)\Theta_{N+2,s}(r-\varepsilon,\rho)\rho^{N+1}d\rho+\displaybreak[1]\\
\notag\\
\notag&\qquad+r^N\int_0^{r-\varepsilon}\big(U(r)-U(\tau)\big)\Theta_{N+2,s}(\tau,\rho)\tau^{N+1}d\tau\Biggr)+\displaybreak[1]\\
\notag\\
\notag&\qquad
+\frac{1}{2\pi}r^N (r-\varepsilon)^{N}\big(U(r)-
U(r-\varepsilon)\big)\Theta_{N+2,s-1}(r-\varepsilon,r).
\end{align}
We first observe that for the last term in \eqref{quattr}, in view of \eqref{uno} and of the fact that $U$ is locally Lipschitz continuous in $\Omega^{\#}$, it holds:
\begin{align*}\allowdisplaybreaks
&r^N (r-\varepsilon)^{N}\big(U(r)-
U(r-\varepsilon)\big)\Theta_{N+2,s-1}(r-\varepsilon,r)\xrightarrow{\varepsilon\rightarrow0}0
\end{align*}
As regards the two integrals in \eqref{quattr} we  change the integration variable to get
\begin{align}\allowdisplaybreaks\label{quatt}
&(r-\varepsilon)^N
\int_{r}^{+\infty}\big(U(r-\varepsilon)-U(\rho)\big)\Theta_{N+2,s}(r-\varepsilon,\rho)\rho^{N+1}d\rho+\displaybreak[1]\\
\notag\\
\notag&\qquad\qquad+r^N\int_0^{r-\varepsilon}\big(U(r)-U(\tau)\big)\Theta_{N+2,s}(\tau,\rho)\tau^{N+1}d\tau=\displaybreak[1]\\
\notag\\
\notag&\qquad={(r-\varepsilon)^N}{r^{N+2}}\int_1^{+\infty}\big(U(r-\varepsilon)-U(\rho\,{r})\big)\Theta_{N+2,s}(r-\varepsilon,\rho\,{r})\rho^{N+1}d\rho+\displaybreak[1]\\
\notag\\
\notag&\qquad\qquad+{r^N}{(r-\varepsilon)^{N+2}}\int_1^{+\infty}\big(U(r)-U(\tfrac{r-\varepsilon}\rho)\big)\Theta_{N+2,s}(\tfrac{r-\varepsilon}\rho,r)\rho^{-N-3}d\rho.
\end{align}
We observe that  $U$ is of class $C^{1,1}$, so, for $\rho\ge1$ it holds
\begin{equation*}
0\le U(r-\varepsilon)-U(\rho\,r)=U'(\eta_{\varepsilon}(\rho))\big(r-\varepsilon-\rho\,r\big),\qquad\text{ for some }\eta_\varepsilon(\rho)\in \big(r-\varepsilon,\rho\,r\big)
\end{equation*}
and
\begin{equation*}
0\le U(\tfrac{r-\varepsilon}\rho)-U(r)=U'(\sigma_\varepsilon(\rho))\big(\tfrac{r-\varepsilon}\rho-r\big),\qquad\text{ for some }\sigma_\varepsilon(\rho)\in \big(\tfrac{r-\varepsilon}\rho,r\big).
\end{equation*}
By the definition of $\Theta_{N+2,s}$ we can use the homogeneity property
\begin{equation*}
\Theta_{N+2,s}(\tfrac{r-\varepsilon}\rho,r)=\rho^{N+2+2s}\Theta_{N+2,s}({r-\varepsilon},\rho\,r)
\end{equation*}
and \eqref{quatt} becomes
\begin{align*}\allowdisplaybreaks\label{quatt}
&(r-\varepsilon)^N
\int_{r}^{+\infty}\big(U(r-\varepsilon)-U(\rho)\big)\Theta_{N+2,s}(r-\varepsilon,\rho)\rho^{N+1}d\rho+\displaybreak[1]\\
\notag\\
\notag&\qquad\qquad+r^N\int_0^{r-\varepsilon}\big(U(r)-U(\tau)\big)\Theta_{N+2,s}(\tau,\rho)\tau^{N+1}d\tau=\displaybreak[1]\\
\notag\\
\notag&\qquad={(r-\varepsilon)^N}{r^{N+2}}\int_1^{+\infty}\left(U'(\eta_h(\rho))-U'(\sigma_h(\rho))\frac{(r-\varepsilon)^2}{r^{2}}\rho^{-N-3+2s}\right)\big(r-\varepsilon-\rho\,r\big)\times \displaybreak[1]\\
\notag\\
&\hskip10cm\times\Theta_{N+2,s}(r-\varepsilon,\rho\,{r})\rho^{N+1}d\rho=\displaybreak[1]\\
\notag\\
\notag&\qquad={(r-\varepsilon)^N}{r^{N+2}}\int_1^{+\infty}\big(U'(\eta_h(\rho))-U'(\sigma_h(\rho))\big)\big(r-\varepsilon-\rho\,r\big)\times \displaybreak[1]\\
\notag\\
&\hskip10cm\times\Theta_{N+2,s}(r-\varepsilon,\rho\,{r})\rho^{N+1}d\rho+\displaybreak[1]\\
\notag\\
\notag&\qquad\qquad+{(r-\varepsilon)^N}{r^{N+2}}\int_1^{+\infty}U'(\sigma_h(\rho))\left(1-\frac{(r-\varepsilon)^2}{r^{2}}\rho^{-N-3+2s}\right)\big(r-\varepsilon-\rho\,r\big)\times \displaybreak[1]\\
\notag\\
&\hskip10cm\times\Theta_{N+2,s}(r-\varepsilon,\rho\,{r})\rho^{N+1}d\rho.
\end{align*}
In order to pass to the limit as $\varepsilon\rightarrow0$ we need to know the behaviour of the integrands when $\rho$ is close to 1. Using \eqref{uno} and the fact that $U'$ is locally Lipschitz continuous in $\Omega^{\#}$, we have, for $\varepsilon$ small enough and for $\rho-1$ small enough
\begin{equation*}
\big|\big(U'(\eta_h(\rho))-U'(\sigma_h(\rho))\big)\big(r-\varepsilon-\rho\,r\big)\Theta_{N+2,s}(r-\varepsilon,\rho\,{r})\big|\le c\big|r-\varepsilon-\rho\,r\big|^{1-2s}
\end{equation*}
and, observing that
\begin{align*}
&\big|r-\varepsilon-\rho\,r\big|^{1-2s}\le c\qquad\text{ if }0< s\le\frac12
\displaybreak[1]\\
\notag\\
\notag&
\big|r-\varepsilon-\rho\,r\big|^{1-2s}\le c(r)(\rho-1)^{1-2s}\qquad\text{ if }\frac12< s<1
\end{align*}
it follows that the first integrand is dominated by a summable function in a right neighborhood of $\rho=1$. For what concerns the summability in a neighborhood of $+\infty$ it is enough to observe that from the asymptotic behavior \eqref{infty} we find, for $\rho\rightarrow +\infty$
\[
\Theta(r-\varepsilon,\rho r)\rho^{N+1}\sim C(r)\frac{1}{\rho^{2s+1}}.
\]
As regards the second integrand, we have for $\varepsilon$ small enough and for $\rho-1$ small enough
\begin{align*}
&\left|U'(\sigma_h(\rho))\left(1-\frac{(r-\varepsilon)^2}{r^{2}}\rho^{-N-3+2s}\right)\big(r-\varepsilon-\rho\,r\big)
\Theta_{N+2,s}(r-\varepsilon,\rho\,{r})\right|\le
\displaybreak[1]\\
\notag\\
\notag&\qquad
\le c\big|r-\varepsilon-\rho\,r\big|^{-2s}\left(\rho^{N+3-2s}-\frac{(r-\varepsilon)^2}{r^{2}}\right)\le c(r)
\big|r-\varepsilon-\rho\,r\big|^{-2s}\left(r\rho^{\frac{N+3-2s}2}-(r-\varepsilon)\right).
\end{align*}
We observe that the function
\begin{equation*}
x\rightarrow\frac{ r\rho^{\frac{N+3-2s}2}-x}{\big(\rho\,r-x\big)^{2s}}
\end{equation*}
is increasing with respect to $x\in\big(0,r\big)$, for every fixed $\rho>1$ and $r>0$. Then it follows
\begin{align*}
&\big|r-\varepsilon-\rho\,r\big|^{-2s}\left(r\rho^{\frac{N+3-2s}2}-(r-\varepsilon)\right)\le C(r)
\frac{\rho^{\frac{N+3-2s}2}-1}{(\rho-1)^{2s}}
\end{align*}
and then the second integrand is dominated by a summable function in a right neighborhood of $\rho=1$: indeed, this is clear for $s<1/2$, while for $s\geq1/2$ we have
\[
\frac{\rho^{\frac{N+3-2s}2}-1}{(\rho-1)^{2s}}\leq {C(r)}{(\rho-1)^{1-2s}}.
\]

Hence from \eqref{quattr}, \eqref{quatt} we have from the homogeneity of $\Theta_{N+2,s}$
\begin{equation*}
\mathsf{I}_{t,\varepsilon}\xrightarrow{\varepsilon\rightarrow0}\frac{N+2s}{2\pi}{r^{N-2s}}\int_1^{+\infty}\Big(U(r)-U(\rho\, r)
+\big(U(r)-U(\tfrac{r}\rho)
\big)\rho^{-N-2+2s}\Big)\Theta_{N+2,s}(1,\rho)\rho^{N+1}d\rho
\end{equation*}
which is proportional to the $s$-Laplacian of the function $U(x)=U(|x|)$ computed in $\R^{N+2}$ at the point $r$ (see \cite{FV}). More precisely, observing that
\begin{equation*}
\gamma(N,s)\frac{N+2s}{2\pi}=\frac{s2^{2s}\Gamma\big(\frac{N+2s}2\big)}{\pi^{\frac N2}\Gamma(1-s)}\frac{N+2s}{2\pi}=\gamma(N+2,s)
\end{equation*}
and
using inequality \eqref{inequality} we get
\begin{align*}\allowdisplaybreaks
\frac{\gamma(N+2,s)}{r^{2s}}&\int_1^{+\infty}\Big(U(r)-U(\rho\, r)
+\big(U(r)-U(\tfrac{r}\rho)
\big)\rho^{-N-2+2s}\Big)\Theta_{N+2,s}(1,\rho)\rho^{N+1}d\rho\displaybreak[1]\\
\\
&\le
\frac{1}{r^N}\int_0^{r}f^{\ast}(\omega_{N}\rho^{N})\rho^{N-1}d\rho
\end{align*}
and from \cite[Theorem 1]{FV}
\begin{equation}
(-\Delta)_{\R^{N+2}}^{s}U(r)\leq \frac{1}{r^N}\int_0^{r}f^{\ast}(\omega_{N}\rho^{N})\rho^{N-1}d\rho \label{comparison}
\end{equation}
for all positive $r$.

\vskip6pt
\noindent $\bullet$ {\textit{Step 4: Comparison principle and end of the proof}}.

\vskip6pt
\noindent Now, for what concerns the solution $v$ to the symmetrized problem \eqref{eq.1} we notice that inequality \eqref{Polyatype} becomes an \emph{equality} for the radial symmetry, thus instead of \eqref{comparison} we find
\begin{equation*}
(-\Delta)_{\R^{N+2}}^{s}V(r)= \frac{1}{r^N}\int_0^{r}f^{\ast}(\omega_{N}\rho^{N})\rho^{N-1}d\rho \label{comparison2}
\end{equation*}
where $V(r)$ is the spherical mean of $v$, \emph{i.e.}
\[
V(x)=V(|x|)=\frac1{|x|^N}\int_0^{|x|}v(\rho)\rho^{N-1}d\rho.
\]
Then we reach to the following crucial estimate
\begin{equation}
(-\Delta)_{\R^{N+2}}^{s}U(r)\leq (-\Delta)_{\R^{N+2}}^{s}V(r)\label{comparisonLaplacian}
\end{equation}
on the \emph{whole} space $\R^{N+2}$, equipped with decay conditions for $U,\,V$, that is $U,\,V\rightarrow0$ as $r=|x|_{N+2}\rightarrow\infty$. We claim that
\begin{equation}
U\leq V.\label{UV}
\end{equation}
Indeed, assume that $W:=U-V>0$ is some point $x_{0}$. Let $W(\bar{x}):=\max W>0$, then by \eqref{comparisonLaplacian} and the very definition of fractional Laplacian
\[
0\leq(-\Delta)_{\R^{N+2}}^{s}W(\bar{x})\leq0
\]
thus $(-\Delta)_{\R^{N+2}}^{s}W(\bar{x})=0$, but this implies $W=W(\bar{x})$, and the decay assumption $W\rightarrow0$ for $|x|_{N+2}\rightarrow\infty$ yields $W(\bar{x})=0$, a contradiction. Then \eqref{UV} holds, namely
\[
u^{\#}\prec v.
\]

\subsection{The general case}
Now we remove the hypotheses made in the previous subsection.
If $f(x)\ge0$ is such that $f\in L^p(\Omega)$, $p\ge 2N/(N + 2s)$, with $N>2s$, we consider, for a sequence of smooth compactly supported functions $f_{n}$ and we denote by $u_n$ the corresponding solutions to problem \eqref{eq.0} with data $f_{n}$ and by $v_{n}$ the solutions to the symmetrized problems \eqref{eq.1} with data $f_{n}^{\#}$. {It is not difficult to prove that $u_{n}\rightharpoonup u$ weakly in $\mathcal{H}^{s}(\Omega)$ and $v_{n}\rightharpoonup v$ weakly in $\mathcal{H}^{s}(\Omega^{\#})$. For instance, taking $u_{n}$ as a test function in the weak formulation of problem \eqref{eq.0} with datum $f_{n}$ one has the energy inequality
\begin{align*}
\frac{\gamma(N,s)}{2}\int_{\R^{N}}\int_{\R^{N}}\frac{|u_{n}(x)-u_{n}(y)|^{2}}{|x-y|^{N+2s}}dx\,dy&=
\int_{\Omega}f_{n}(x)\,u_{n}(x)\,dx\\&\leq\|f_{n}\|_{L^{(2^{*}_s)^{\prime}}(\Omega)}\|u_{n}\|_{L^{2^{*}_s}(\Omega)}
\end{align*}
where $2^{*}_s=2N/(N-2s)$ and $(2^{*}_s)^{\prime}=2N/(N+2s)$. Thus fractional Sobolev and Rellich-Kondrachov Theorem (see, \emph{e.g.}, \cite{hitch}) imply that up to subsequences $u_{n}\rightharpoonup u$ weak in $\mathcal{H}^{s}(\Omega)$ and $u_{n}\rightarrow u$ strong in $L^{q}(\Omega)$ for all  $q<2N/(N-2s)$. On the other hand, if $N=1$ and $s\ge1/2$, we suppose $p>1$ and we can use the fact that $u$ and $u_n$ belong to $L^q(\Omega)$ for every $q<+\infty$, obtaining again that up to subsequences $u_{n}\rightharpoonup u$ weak in $\mathcal{H}^{s}(\Omega)$ and $u_{n}\rightarrow u$ strong in $L^{q}(\Omega)$ for all  $q<+\infty$. This is enough to pass in the weak formulation satisfied by $u_n$. A similar argument can be done for $v_n$.} Now, by the previous subsection we have
\[
u_n^{\#}\prec v_n
\]
%
%
and, passing to the limit in $n$, we have our concentration estimate \eqref{UV} when $f(x)\ge0$.

Finally, if no sign assumption is made on $f$ we observe that the comparison principle implies $|u|\leq \tilde u$, being $\tilde u$ the solution to the elliptic problem \eqref{eq.0} having $|f|$ as source datum.
Thus, applying \eqref{UV} to $\tilde u$, we have
\[
u^{\#}\prec\tilde u^\#\prec v,
\]
and the theorem is completely proved.
\hfill\qed

\section{Extensions and remarks}\label{Sec5}

The methods used in the present note appear to be suitable for the investigation about the effects of symmetrization on the solutions of various classes of nonlocal PDEs, such as semilinear equations, fractional parabolic equations of porous medium type (particularly the ones in bounded domains),
equations involving operators (mentioned in the Introduction) with general L\'{e}vy kernels or the nonlinear variant of the fractional Laplacian, the so-called fractional $p$-Laplacian. We plan to address these topics in forthcoming papers.

Here we just point out that an almost immediate application of our main result allows us to state a symmetrization result for linear equations with a zero-order term, namely
\begin{equation}
\left\{
\begin{array}
[c]{lll}%
( -\Delta) ^{s}u+c\,u=f & & \text{in }%
\Omega\\
\\
v=0 & & \text{on }\R^{N}\setminus\Omega.
\end{array}
\right. \label{eq.c}
\end{equation}
where, for example, $c=c(x)\ge0$ with $c\in L^{\infty}(\Omega)$. Indeed, because of the sign assumption, the coefficient $c$ has no influence when the solution $u$ is compared with the solution $v$ to the symmetrized problem \eqref{eq.1}.

A different story is when we wish \emph{not} to neglect the coefficient $c$ in the symmetrization procedure. For instance, assume that $c>0$ is \emph{constant} and we want to compare $u$ with the solution $v$ to the problem
 \begin{equation}
\left\{
\begin{array}
[c]{lll}%
\left( -\Delta\right) ^{s}v+cv=f^{\#}\left( x\right) & & \text{in }%
\Omega^{\#}\\
\\
v=0 & & \text{on }\R^{N}\setminus\Omega^{\#}.
\end{array}
\right. \label{eq.cs}%
\end{equation}
With the same arguments of the proof of Theorem \ref{main} we arrive to the following inequality, satisfied by $u$:
\begin{align*}
\frac{\gamma(N,s)}{N}
\int_0^{r(t)^N}\left(\int_{r(t)^N}^{+\infty}\big(\mathfrc{u}(\mathfrc{r}^{\frac1N})-\mathfrc{u}(\mathfrc{s}^{\frac1N})\big)\Theta_{N,s}(\mathfrc{r}^{\frac1N},\mathfrc{s}^{\frac1N})d\mathfrc{s}\right)d\mathfrc{r}
&+\frac{1}{N}\int_{0}^{r(t)^{N}}\mathfrc{u}(\mathfrc{r}^{\frac1N})d\mathfrc{r}\\
&\le \int_0^{r(t)^N}f^{*}(\omega_{N}\mathfrc{s}^{\frac1N})d\mathfrc{s}
\end{align*}
Observe now that if $u$ has a flat zone at level $t>0$ we easily find, for all $r(t)^{N}<\sigma<r(t^{-})^{N}$,
\[
\int_{0}^{\sigma}\mathfrc{u}(\mathfrc{r}^{\frac1N})d\mathfrc{r}=\int_{0}^{r(t)^{N}}\mathfrc{u}(\mathfrc{r}^{\frac1N})d\mathfrc{r}
+(\sigma-r(t)^{N})t
\]
which is a linear function in $\sigma$, thus the same convexity argument in the proof of Theorem \ref{main} provides
\begin{align*}\label{inequalityCC}
\gamma(N,s)
\int_0^{r}\left(\int_{r}^{+\infty}\big(\mathfrc{u}(\tau)-\mathfrc{u}(\rho)\big)\Theta_{N,s}(\tau,\rho)\rho^{N-1}d\rho\right)\tau^{N-1}d\tau
&+c\int_{0}^{r}\mathfrc{u}(\tau)\tau^{N-1}d\tau\\&
\le\int_0^{r}f^{*}(\omega_{N}\rho^{N})\rho^{N-1}d\rho,
\end{align*}
which becomes an equality when replacing $u$ with $v$. Then we can argue as in Theorem \ref{main} and \eqref{compariso} holds with $u$ and $v$ solutions to \eqref{eq.c} and \eqref{eq.cs}, respectively.
\vskip7pt
We conclude this section with a few remarks about the three main inequalities derived in the proof of the main theorem, namely, inequalities \eqref{Polyatype}, \eqref{inequality} and \eqref{comparison}.

\begin{remark}
Inequality \eqref{Polyatype} allows us to deduce \eqref{inequality}, that is, the ``fractional'' counterpart of inequality \eqref{first} which holds true for solutions to problem \eqref{eq.local}. We would like to emphasize that, unlike the fractional case, in the local case the derivation of \eqref{first} is quite natural because the gradient of a truncation of the solution can be easily computed. Indeed, if $z$ is a solution to \eqref{eq.local}, the use of the test function $\mathcal{G}_{t,h}(z)$ and of the ellipticity condition gives:
$$\int_{\Omega}|D\mathcal{G}_{t,h}(z)|^2dx\le\int_\Omega f(x)\mathcal{G}_{t,h}(z(x))\,dx.
$$
The use of P\'olya-Szeg\"o principle is quite immediate because the function $\mathcal{G}_{t,h}(z)$ is a Sobolev function and it immediately follows
$$\int_{\Omega^\#}|D\mathcal{G}_{t,h}(z^\#)|^2dx\le\int_{\Omega^\#} f^\#(x)\mathcal{G}_{t,h}(z^\#(x))\,dx,
$$
that is, the analogous of inequality \eqref{Polyatype} for the solution of problem \eqref{eq.local}.
\end{remark}
{\begin{remark}
According to Remark \ref{Limit s}, the solution $u$ to problem \eqref{eq.0} for $s=1$, \emph{i.e.} the solution to the local Poisson equation with homogeneous boundary condition $u=0$ on $\partial \Omega$, can be seen as the weak limit of the family of functions $\mathfrc{u}_{s}=u^\#_s$ to \eqref{eq.0} for $s\in (0,1)$.  Observe that the left-hand side of \eqref{inequality} can be written in the form
\begin{align*}
\gamma(N,s)
&\int_0^{r}\left(\int_{r}^{+\infty}\big(\mathfrc{u}_s(\tau)-\mathfrc{u}_s(\rho)\big)\Theta_{N,s}(\tau,\rho)\rho^{N-1}d\rho\right)\tau^{N-1}d\tau\\
&\hskip3cm
=
\gamma(N,s)\int_{B_r}\left(\int_{\R^N}\frac{\mathfrc{u}_{s}(x)-\mathfrc{u}_{s}(y)}{|x-y|^{N+2s}}dy\right)dx=\int_{B_r}(-\Delta)^s \mathfrc{u}_{s}\,dx
\end{align*}
then, passing to the limit as $s\rightarrow1$, the divergence theorem gives
\[
-N\omega_{N}r^{N-1}\mathfrc{u}'(r)\le\int_0^{r}f^{*}(\omega_{N}\rho^{N})\rho^{N-1}d\rho,
\]
which is an equality when $u$ is replaced by $v$. Then integrating on $(0,|\Omega|)$ and using the zero boundary conditions gives
\begin{equation}\label{pointwise}
u^{\#}(x)\leq v(x),\quad x\in \Omega^{\#}
\end{equation}
namely the classical pointwise Talenti's inequality. \\ Another easier form to recover the pointwise comparison is observing that letting $s\rightarrow1$ in \eqref{comparisonLaplacian}, which is a consequence of \eqref{comparison}, provides a comparison between local Laplacians
\[
(-\Delta)_{\R^{N+2}}U(r)\leq (-\Delta)_{\R^{N+2}}V(r)
\]
and a straightforward computation shows
\[
(-\Delta)_{\R^{N+2}}U(r)=-\frac{{u^\#}^{\prime}(r)}{r},\quad(-\Delta)_{\R^{N+2}}V(r)=-\frac{v^{\prime}(r)}{r}
\]
then we recover \eqref{pointwise} again.
\end{remark}}

\begin{remark} A way to get \eqref{comparison} is to use the representation via Fourier transform of the fractional Laplacian applied to radial functions as described in Section \ref{Sec2}. Actually, the idea to deduce an inequality written in terms of a fractional Laplacian computed in $\R^{N+2}$ has originated from the computations we give here. In what follows we suppose that all the passages are justified, in particular we suppose that $\mathfrc{u}(x)=u^\#(x)$ is regular enough in such a way that one can compute $(-\Delta)^s\mathfrc{u}$ pointwise, a property which does not need to be satisfied in our context. Our aim is to compute the integral which is on the left-hand side of \eqref{inequality}, that is,
\begin{equation}\label{integ}
Y(r)=\gamma(N,s)
\int_0^{r}\left(\int_{r}^{+\infty}\big(\mathfrc{u}(\tau)-\mathfrc{u}(\rho)\big)\Theta_{N,s}(\tau,\rho)\rho^{N-1}d\rho\right)\tau^{N-1}d\tau.
\end{equation}
Using the very definition of fractional Laplacian we have
$$Y(r)=
\int_0^{r}(-\Delta)^s\mathfrc{u}(\tau)
\tau^{N-1}d\tau
$$
and, by Theorem \ref{bessela}, we get
$$Y(r)=(2\pi)^{2s+2}
\int_0^{r}
\left(\int_{0}^{+\infty}
\rho^{1+2s}J_{\frac N2-1} (2\pi\tau\rho)\left(\int_{0}^{+\infty}\sigma^{\frac N2}\mathfrc{u}(\sigma)J_{\frac N2-1} (2\pi\rho\sigma)\,d\sigma\right)d\rho\right)
\tau^{\frac N2}d\tau.
$$
Now, supposing that we can do it, we integrate by parts, exchange the order of integration and use the following property of Bessel functions
$$x^{1-\frac N2}J_{\frac N2}(2\pi x\tau)=\frac1{2\pi\tau}\frac d{dx}\left(x^{1-\frac N2}J_{\frac N2-1}(2\pi x\tau)\right),$$
to get
\begin{align*}\allowdisplaybreaks
Y(r)=&(2\pi)^{2s+3}
\int_0^{r}
\left(\int_{0}^{+\infty}
\rho^{2+2s}J_{\frac N2-1} (2\pi\tau\rho)\left(\int_{0}^{+\infty}\sigma^{1+\frac N2}U(\sigma)J_{\frac N2} (2\pi\rho\sigma)\,d\sigma\right)d\rho\right)
\tau^{\frac N2}d\tau
\displaybreak[1]\\
\\
=&(2\pi)^{2s+2}r^{\frac N2}
\int_{0}^{+\infty}
\rho^{1+2s}J_{\frac N2} (2\pi\tau\rho)\left(\int_{0}^{+\infty}\sigma^{1+\frac N2}U(\sigma)J_{\frac N2-1} (2\pi\rho\sigma)\,d\sigma\right)d\rho.
\end{align*}
Using again Theorem \ref{bessela} to compute the $s$-Laplacian in $\R^{N+2}$ applied to $U$, we have:
$$Y(r)=r^N(-\Delta)^s_{\R^{N+2}}U(r).
$$
From \eqref{inequality} it follows \eqref{comparison}.
\end{remark}

\section*{Acknowledgments}

V.F. was partially supported by Italian MIUR through research project PRIN 2017 ``Direct and inverse problems for partial differential equations: theoretical aspects and applications''. B.V. was partially supported by Gruppo Nazionale per l'Analisi Matematica, la Probabilit\`a e le loro Applicazioni (GNAMPA) of Istituto Nazionale di Alta Matematica (INdAM). Both authors are members of GNAMPA of INdAM.\nc

\bibliographystyle{siam}\small
\bibliography{FerVol_arXiv}

\

2000 \textit{Mathematics Subject Classification.}
35B45,  
35R11,   	
35J25. 


%
\textit{Keywords and phrases.} Symmetrization, fractional Laplacian,
 nonlocal elliptic equations.

\end{document}